\documentclass[11pt]{amsart}

\usepackage{amsmath,amssymb,amsthm}
\usepackage{mathrsfs}
\usepackage{mathtools}
\usepackage{microtype}
\usepackage{geometry}
\usepackage{booktabs}
\usepackage[colorlinks=true,linkcolor=blue,citecolor=blue,urlcolor=blue]{hyperref}
\usepackage{cleveref}
\newtheorem{theorem}{Theorem}[section]
\newtheorem{proposition}[theorem]{Proposition}
\newtheorem{lemma}[theorem]{Lemma}
\newtheorem{corollary}[theorem]{Corollary}
\theoremstyle{definition}
\newtheorem{definition}[theorem]{Definition}

\theoremstyle{remark}
\newtheorem{remark}[theorem]{Remark}
\newtheorem{notation}[theorem]{Notation}

\newcommand{\C}{\mathbb{C}}
\newcommand{\R}{\mathbb{R}}
\newcommand{\Z}{\mathbb{Z}}

\newcommand{\g}{\mathfrak{g}}

\newcommand{\Der}{\mathrm{Der}}
\newcommand{\pa}{\partial}
\newcommand{\hDer}{\widehat{\Der(A)}}
\newcommand{\hg}{\widehat{\g\otimes A}}
\newcommand{\ox}{\otimes}
\newcommand{\ol}[1]{\overline{#1}}
\newcommand{\sltwo}{\mathfrak{sl}_2}

\DeclareMathOperator{\Span}{Span}

\title[Orthogonal polynomials and superelliptic UCE centres]
{Orthogonal polynomials and the centres of the universal\\
central extensions of superelliptic Krichever--Novikov algebras}

\author[Albino dos Santos]{Felipe Albino dos Santos}
\address{Universidade Presbiteriana Mackenzie,
Faculdade de Computa\c{c}\~ao e Inform\'atica, Brazil}
\email{falbinosantos@gmail.com}

\date{\today}

\keywords{superelliptic curve; Krichever--Novikov algebra; universal central
extension; ultraspherical polynomials; associated orthogonal polynomials;
Legendre polynomials; K\"ahler differentials; cocycle}
\subjclass[2020]{17B65, 17B56, 33C45, 14H99, 17B10}

\begin{document}
\begin{abstract}
We study the universal central extensions of the derivation and current
algebras of a superelliptic curve $u^m=P(x)$, for $P$ palindromic of degree
$2r$. Both centres are finite-dimensional and canonically isomorphic. On the
derivation side the dimension is known at every covering degree, while reducing
an arbitrary class has been done only for the double cover $m=2$; we do it at
every $m$. The spanning classes split into $r(m-1)$ chains, each an associated
ultraspherical family whose parameters come from the sector and the residue,
and we determine which chains carry a positive orthogonality measure. For
quadratic $P$ we evaluate the mixed cocycle in closed form, as a combination of
two Legendre antiderivatives, at every pair of generators of total index at
least three.
\end{abstract}

\maketitle
\setcounter{tocdepth}{1}
\tableofcontents

\section{Introduction}
\label{sec:intro}

This paper joins two subjects that developed independently, the classical
theory of orthogonal polynomials and the theory of infinite-dimensional Lie
algebras attached to algebraic curves.

The first goes back to Legendre, who introduced the polynomials now carrying
his name while studying gravitational potentials. The classical families and
their orthogonality relations were codified during the nineteenth century and
brought into modern form in Szeg\H{o}'s monograph \cite{Szego1939}. What makes
a family classical is a pair of equivalent conditions, a three-term recurrence
with coefficients affine in the index and a second-order differential equation
\cite[\S 3.2]{Szego1939}. The Legendre polynomials $p_n$ are the case we shall
meet, characterised by
\[
  (n+1)p_{n+1}(a) = (2n+1)a\,p_n(a) - n\,p_{n-1}(a),
  \qquad p_0(a)=1,\quad p_1(a)=a,
\]
together with the generating function $\sum_{n\geq0}p_n(a)z^n=(1-2az+z^2)^{-1/2}$
\cite[(4.7.23)]{Szego1939}.

The second subject begins with Krichever and Novikov, who introduced algebras
of meromorphic vector fields on a compact Riemann surface with poles confined
to two marked points, as higher-genus analogues of the Witt and Virasoro
algebras \cite{KrichNovikov1987}. Schlichenmaier extended the construction to
finitely many marked points and supplied the almost-graded structure that makes
the representation theory tractable
\cite{Schlichenmaier1990a,Schlichenmaier1990b,Schlichenmaier2003}; a systematic
account is \cite{Schlichenmaier2014}. For a commutative ring $R$ of functions on
such a curve and a finite-dimensional simple Lie algebra $\g$, the current
algebra $\g\otimes R$ is perfect, so it has a
universal central extension, and Kassel and Loday identified that extension and
its centre $\Omega^1_R/dR$, the K\"ahler differentials modulo exact forms
\cite{KasselLoday1982}.

The two subjects meet over the superelliptic curve $u^m=P(x)$, whose coordinate
ring we write $A$ throughout. For the current
algebra, the centre $\Omega^1_A/dA$ was computed and given an explicit basis in
\cite[Theorem~2.4]{Santos2022}. For the derivation algebra of the same
coordinate ring, Cox, Guo, Lu and Zhao showed that every derivation is a
multiple of one distinguished derivation, and computed the dimension of the
corresponding centre $A/\partial A$ at every covering degree
\cite[Lemma~3]{CoxGuoLuZhao2017}, \cite[Theorem~9]{CoxGuoLuZhao2017}. Their
proof produces the differentiation formulas
\cite[(3.1)--(3.2)]{CoxGuoLuZhao2017} and then uses them to count, and the
sentence that closes the argument is the one this paper starts from: knowing the
dimension, one still has to compute the class of an arbitrary $f$, and that,
they write, ``is actually not easy''. It is not carried out there; at $m=2$ it is
carried out by Cox and Zhao \cite{CoxZhao2018}. That orthogonal polynomials appear in these
centres at all was first noticed by Cox, Futorny and Tirao for the
Date--Jimbo--Kashiwara--Miwa (DJKM) algebras \cite{CoxFutornyTirao2013}, and was developed for the superelliptic
current algebra in \cite{SantosNeklyudovFutorny2025}, where the structure
constants of the cocycle were computed in closed form. A representation-theoretic
account of the same phenomenon on that side, in which a Sugawara element acting
on a Verma module of the Heisenberg subalgebra of $\widehat{\g\otimes A}$ is
carried by an explicit intertwiner to the classical Legendre differential
operator, is given for the genus-zero four-point curve in
\cite{P4_RepMechanism}.

On the derivation side the double cover is understood in detail, and the account
this paper builds on is that of Cox and Zhao \cite{CoxZhao2018}. For $u^2=P(x)$
with $P$ of arbitrary degree they compute the reduction relations in $A/\pa A$
and the values of the cocycle on a basis \cite[\S2.1]{CoxZhao2018}; they work the
quadratic case out completely, obtaining the three-dimensional centre, the
Legendre recurrence and the value of the cocycle on every pair of generators
\cite[\S2.2]{CoxZhao2018}; and for the palindromic polynomial
$P(x)=x^{2r}-2ax^{r}+1$ they show that the centre has dimension $2r+1$ and that
the classes $[x^k]$ split by residue modulo $r$ into families of associated
Legendre polynomials \cite[\S2.4]{CoxZhao2018}. Everything the present paper
recalls about $m=2$ is theirs; what it adds at $m=2$ is the evaluation of the
mixed cocycle and the positivity criterion.

What their account leaves open is the covering degree. It uses $u^2=P$ at every
step, through $\pa=u\,d/dx$ and $u\cdot u=P$, and the polynomials it produces are
Legendre because a double cover is what puts $\tfrac12$ in the ultraspherical
parameter. The present paper removes that restriction, and it carries the cocycle
computation one step further than its source. We take $P$ palindromic throughout,
so that the branch locus is symmetric under $x\mapsto1/x$, and we keep the
covering degree general where the argument allows it.

Both centres have the same dimension. Writing $n=\deg P$, one has
$\dim_\C A/\partial A = 1+n(m-1)$ by \cite[Theorem~9]{CoxGuoLuZhao2017} and
$\dim_\C\Omega^1_A/dA = 1+n(m-1)$ by \cite[Theorem~2.4]{Santos2022}. What
distinguishes the two extensions is therefore not the size of the centre. It is
that only the derivation extension carries a cocycle mixing its two families of
generators, and we evaluate that cocycle in closed form.

The mechanism on the derivation side is elementary and worth stating at once.
The distinguished derivation is $\partial=u^{m-1}d/dx$, and in the quotient
$A/\partial A$ every class $[\partial(h)]$ vanishes. Differentiating the
monomials $x^ku^{j}$ therefore produces linear relations among the spanning
classes, and a one-line application of the Leibniz rule, carried out in
Lemma~\ref{lem:partial_general}, shows that each relation involves exactly three
of them and stays inside one sector of the $\Z/m\Z$-grading. For the quadratic
palindromic polynomial $P(x)=x^2-2ax+1$ with $a\neq\pm1$ and $m=2$ this gives
\[
  \partial(x^ku) = (k+1)x^{k+1} - a(2k+1)x^k + k\,x^{k-1},
\]
so that in $A/\partial A$
\begin{equation}\label{eq:rec_intro}
  (k+1)[x^{k+1}] - a(2k+1)[x^k] + k[x^{k-1}] = 0 .
\end{equation}
This is the Legendre recurrence, term for term, with the branch-point
parameter $a$ playing the role of the polynomial variable. Since
$[x^0]$ and $[x^1]=a[x^0]$ match $p_0(a)$ and $p_1(a)$, uniqueness of the
solution of a three-term recurrence forces $[x^k]=p_k(a)[x^0]$ for every
$k\geq0$. The same relation run downwards gives a second, independent chain on
the negative powers, and the classes $[x^ju^{m-1}]$ contribute one further
generator, which accounts for all three dimensions.

At a general covering degree the same computation gives, for
$P=1-2cx^r+x^{2r}$, whose parameter $c$ plays the role that $a$ played above,
and for each $l$ with $0\leq l\leq m-2$,
\begin{equation}\label{eq:rec_intro_general}
  mk\,[x^{k-1}u^{l}] - 2c\,(mk+(l+1)r)\,[x^{k+r-1}u^{l}]
    + (mk+2(l+1)r)\,[x^{k+2r-1}u^{l}] = 0 ,
\end{equation}
while the top sector collapses to the single class $[x^{-1}u^{m-1}]$. The three
exponents are congruent modulo $r$, so the spanning classes fall into
$r(m-1)$ chains indexed by a sector and a residue, and
Theorem~\ref{thm:chains} identifies each of them. Written in monic form, the
chain attached to the sector $l$ and the residue $b$ is the ultraspherical
family of parameter $(l+1)/m$, associated with parameter $(b+1-r)/r$. The
covering degree of the curve is thus visible in the centre as the denominator of
the ultraspherical parameter, the degree of $P$ as twice the number of chains
per sector, and the Legendre family of \eqref{eq:rec_intro} is the single case
$m=2$, $r=1$.

From this identification we obtain the generating function of the classes and
the first-order differential equation it satisfies, and then the evaluation of
the mixed cocycle. Let $e_k=x^k\partial$ and $f_k=x^ku\partial$ span
$\Der(A)$ for $m=2$, and let $\psi$ be the canonical cocycle of the universal
central extension of $\Der(A)$. On the mixed pair we prove in
Theorem~\ref{thm:antiderivative} that
\[
  \psi(e_1,f_s) = \frac{p_{n-1}(a)-p_{n+1}(a)}{2n+1}\,[x^0],
  \qquad n=s+1,
\]
and that the scalar appearing here is exactly $\int_a^1p_n(t)\,dt$. The cocycle
of a Lie algebra thus evaluates a classical definite integral, and the proof
runs through the structure constants of the algebra rather than through
integration. Theorem~\ref{thm:cross_closed} carries this to every pair of
generators, where the value is a combination of the two antiderivatives
$\int_a^1p_n$ and $\int_a^1p_{n-2}$ with explicit rational coefficients. Cox and
Zhao write the five-term expression that these two theorems evaluate.

It may help to say in one place what is proved here and what is recalled. Three
things are proved. Theorem~\ref{thm:chains} describes $A/\pa A$ for
$u^m=1-2cx^r+x^{2r}$ at every $m$ and every $r$, splitting the spanning classes
into $r(m-1)$ chains and identifying each as an associated ultraspherical family
whose two parameters are read off from the sector and the residue; its case
$m=2$ is \cite[\S2.4]{CoxZhao2018}, and what is added is everything above the
double cover together with the positivity criterion in part~(5).
Theorems~\ref{thm:antiderivative} and~\ref{thm:cross_closed} evaluate the mixed
cocycle in closed form. Proposition~\ref{prop:centres_iso} identifies the centres
of the two universal central extensions canonically, which is what makes the two
halves of the paper statements about one object. Everything else below concerning
$m=2$ is recalled from \cite{CoxZhao2018}, with attribution at the point of
use. The Cox--Zhao statements that Section~\ref{sec:cross} starts from are
gathered in Section~\ref{sec:quad_der};
Section~\ref{sec:quartic} does give arguments at $m=2$, as worked instances of
the general-$m$ results.





Section~\ref{sec:setup} fixes the curve, the coordinate ring, the two
extensions and the facts about Legendre polynomials that we use.
Section~\ref{sec:quad_der} collects what is used from \cite{CoxZhao2018} at
$m=2$: the Legendre identification and the two cocycle formulas, stated with
locators. Section~\ref{sec:cross} evaluates the mixed cocycle
in closed form at every pair of generators.
Section~\ref{sec:SNF} carries the argument to the current algebra.
Section~\ref{sec:comparison_dict} compares the two extensions.
Section~\ref{sec:chains} proves the chain decomposition at every covering
degree. Section~\ref{sec:quartic} works out the quartic instance, where
classicality first fails, and isolates the role of palindromic symmetry.

\section{Setup and notation}
\label{sec:setup}

\subsection{The curve and its coordinate ring}
\label{subsec:curve}

Fix an integer $m\geq2$, the \emph{covering degree}, and a polynomial
$P(x)\in\C[x]$ of degree $n$ with no repeated roots and with $P(0)\neq0$. The
equation $u^m=P(x)$ defines a \emph{superelliptic curve} $\mathscr C$, and its
\emph{coordinate ring} is
\[
  A = \C[x^{\pm1},u]\big/(u^m-P(x)),
\]
the ring of regular functions on the affine part of $\mathscr C$ with the
fibres over $x=0$ and $x=\infty$ removed. The ring $A$ is $\Z/m\Z$-graded,
\[
  A=\bigoplus_{l=0}^{m-1}A^l,\qquad A^l=\C[x^{\pm1}]\cdot u^l ,
\]
and this grading is used throughout.

The projection $x\colon\mathscr C\to\mathbb P^1$ presents $\mathscr C$ as an
$m$-sheeted covering of the projective line, ramified over the $n$ roots of $P$
and, when $\gcd(m,n)<m$, over $x=\infty$. Since $P(0)\neq0$ the fibre over
$x=0$ is unramified and consists of $m$ points, while the fibre over
$x=\infty$ consists of $\gcd(m,n)$ points. The algebra $A$ is therefore the
ring of functions regular away from $s=m+\gcd(m,n)$ marked points of the
smooth model $X$, so it is a Krichever--Novikov algebra in the multipoint
framework of Schlichenmaier
\cite{Schlichenmaier1990a,Schlichenmaier1990b,Schlichenmaier2014}. By Riemann--Hurwitz applied to this covering the genus of $X$ is
\[
  g = \tfrac12\bigl[\,m(n-1)-n-\gcd(m,n)\,\bigr]+1 .
\]

A polynomial $P(x)\in\C[x]$ of degree $n$ is \emph{palindromic} (or
\emph{self-reciprocal}) if $P(x)=x^{n}P(1/x)$; equivalently, its coefficient
sequence is symmetric. Except where a general $P$ is explicitly allowed, we work
throughout with the palindromic polynomial
\begin{equation}\label{eq:palindromic_P}
  P(x) = 1 - 2c\,x^r + x^{2r},\qquad c\in\C,\quad c\neq\pm1,\quad r\geq1,
\end{equation}
of degree $n=2r$, which has no repeated roots for $c\neq\pm1$ and satisfies
$P(0)=1\neq0$.

For $m=2$ the marked-point count is $s=2+\gcd(2,2r)=4$ for \emph{every} $r$, and
the genus is $g=r-1$. So the whole palindromic family at $m=2$ consists of
four-point algebras, of which only $r=1$ is the genus-zero four-point algebra of
Bremner \cite[Prop.~1.1]{Bremner1995}; the quartic case $r=2$ is a four-point
algebra on a curve of genus one, and $r\geq3$ gives four points on a curve of
genus $r-1$. The four-point count is a feature of $m=2$.

The decomposition of the centres proved in Section~\ref{sec:chains} holds
for every $m$ and every $r$, as do the setup and the results of
Section~\ref{sec:SNF}. The evaluation of the cocycle in
Section~\ref{sec:cross} and the worked cases $r=1$ and $r=2$ of
Sections~\ref{sec:quad_der} and~\ref{sec:quartic} are at $m=2$, and the
hypotheses are stated at each occurrence. In the quadratic case we write $P(x)=x^2-2ax+1$, which is
\eqref{eq:palindromic_P} at $r=1$; the parameter $c$ of
\eqref{eq:palindromic_P} and the parameter $a$ of the quadratic normalisation
are the same quantity, and we use whichever letter the surrounding computation
makes natural. The parameter $a$ is then the half-sum of the two roots of $P$, that is,
the point the branch locus is centred on.

\subsection{The distinguished derivation}

\begin{definition}\label{def:partial}
The \emph{distinguished derivation} of $A$ is
\[
  \pa \coloneqq u^{m-1}\frac{d}{dx}\in\Der(A).
\]
\end{definition}

That $\pa$ is well defined requires $\pa(u^m-P(x))=0$, which holds because
$u^m=P(x)$ gives $m\,u^{m-1}\,du = P'(x)\,dx$ and hence
$\pa(u)=u^{m-1}\,du/dx = P'(x)/m$. For $m=2$ this reads $\pa=u\,d/dx$ and
$\pa(u)=P'(x)/2$. Equivalently, $\pa$ is the derivation of $A$ determined by
\begin{equation}\label{eq:partial_values}
  \pa(x)=u^{m-1},\qquad \pa(u)=\tfrac1m P'(x),
\end{equation}
and this is the form in which we shall use it.

The reduction relations of the whole paper come from applying $\pa$ to the
monomials $x^ku^j$ and reading the result in $A/\pa A$. We record that
computation once, at every $m$.

\begin{lemma}[{\cite[(3.1)--(3.2)]{CoxGuoLuZhao2017}}]\label{lem:partial_general}
Let $A=\C[x^{\pm1},u]/(u^m-P)$ and let $1\leq j\leq m-1$. Then for every
$k\in\Z$
\begin{equation}\label{eq:partial_general}
  \pa(x^ku^j) \;=\; u^{j-1}\Bigl(k\,x^{k-1}P(x) + \tfrac{j}{m}\,x^kP'(x)\Bigr),
  \qquad
  \pa(x^k) \;=\; k\,x^{k-1}u^{m-1} .
\end{equation}
In particular $\pa(x^ku^j)$ lies in the sector $A^{\,j-1}$ of the
$\Z/m\Z$-grading, and $\pa(x^k)$ lies in $A^{\,m-1}$.
\end{lemma}

These are equations (3.1) and (3.2) of \cite{CoxGuoLuZhao2017}, written with
$u^{j}$ in place of $P^{j/m}$. We recall the derivation because every relation
below is read off from them.

\noindent
For the palindromic polynomials of this paper the two identities take the
following shape.

\begin{corollary}\label{cor:chain_general}
Let $P(x)=1-2cx^{r}+x^{2r}$ with $r\geq1$. Then in $A/\pa A$
\begin{equation}\label{eq:top_sector}
  [x^{k}u^{\,m-1}] = 0 \qquad (k\neq-1),
\end{equation}
and for every $l$ with $0\leq l\leq m-2$, writing $L=l+1$,
\begin{equation}\label{eq:chain_general}
  mk\,[x^{k-1}u^{l}] \;-\; 2c\,(mk+Lr)\,[x^{k+r-1}u^{l}]
  \;+\; (mk+2Lr)\,[x^{k+2r-1}u^{l}] \;=\; 0,
  \qquad k\in\Z .
\end{equation}
\end{corollary}

\begin{proof}
Since $\pa(x^{k+1})=(k+1)x^ku^{m-1}$ lies in $\pa A$, its class vanishes, which
is \eqref{eq:top_sector}. For \eqref{eq:chain_general} put $j=l+1$ in
\eqref{eq:partial_general} and substitute $P'(x)=-2crx^{r-1}+2rx^{2r-1}$,
\[
  \pa(x^ku^{\,l+1}) = u^{l}\Bigl(k\,x^{k-1} - 2c\bigl(k+\tfrac{Lr}{m}\bigr)x^{k+r-1}
   + \bigl(k+\tfrac{2Lr}{m}\bigr)x^{k+2r-1}\Bigr) .
\]
The left-hand side lies in $\pa A$, so the class of the right-hand side
vanishes; multiplying by $m$ gives \eqref{eq:chain_general}.
\end{proof}

The three exponents in \eqref{eq:chain_general} are congruent modulo $r$, so no
relation couples different residues, and no relation couples different sectors
of the $\Z/m\Z$-grading.

\begin{theorem}\label{thm:der_module}
{\rm\cite[Lemma~3]{CoxGuoLuZhao2017}}
$\Der(A)=A\,\pa$ as an $A$-module, so every derivation of $A$ is $f\pa$ for a
unique $f\in A$.
\end{theorem}

\begin{notation}\label{not:basis}
For $m=2$ we set $e_k \coloneqq x^k\pa$ and $f_k \coloneqq x^k u\pa$ for
$k\in\Z$. By Theorem~\ref{thm:der_module} these span $\Der(A)$ over $\C$, the
$e_k$ spanning the even part of the $\Z/2\Z$-grading and the $f_k$ the odd
part.
\end{notation}

\subsection{The two universal central extensions}

Both Lie algebras of interest are perfect, so each has a universal central
extension, unique up to isomorphism.

\begin{theorem}\label{thm:UCE_Der}
{\rm\cite[Theorem~9]{CoxGuoLuZhao2017}}
Let $P(x)=x^{l}(x-a_1)\cdots(x-a_n)$ with the $a_i\in\C^{*}$ pairwise
distinct and $l\in\Z_{\geq0}$. The universal central extension of $\Der(A)$ is
$\hDer=\Der(A)\oplus(A/\pa A)$ with bracket
\[
  [f\pa,\,g\pa] = \bigl(f\pa(g)-g\pa(f)\bigr)\pa + \ol{\pa(f)\,\pa^2(g)},
\]
(in the setting of \S\ref{subsec:curve} we have $P(0)\neq0$, so $l=0$, the
$a_i$ are the roots of $P$, and $n=\deg P$ as elsewhere in this paper)
and $\dim_\C A/\pa A = 1+n(m-1)$.
\end{theorem}

The canonical $2$-cocycle of this extension is
$\psi\colon\Der(A)\times\Der(A)\to A/\pa A$ given by
\begin{equation}\label{eq:psi_def}
  \psi(f\pa,\,g\pa) = \ol{\pa(f)\,\pa^2(g)} .
\end{equation}
Its values on pairs $e_k,f_s$ drawn from opposite parities are what
Section~\ref{sec:cross} computes.

Let $\g$ be a finite-dimensional simple Lie algebra with invariant bilinear
form $(\cdot,\cdot)$.

\begin{theorem}\label{thm:UCE_gA}
{\rm\cite{KasselLoday1982}}
The current algebra $\g\otimes A$ is perfect, and its universal central
extension is $\hg=(\g\otimes A)\oplus(\Omega^1_A/dA)$ with bracket
\[
  [X\ox f,\,Y\ox g] = [X,Y]\ox fg + (X,Y)\,\ol{f\,dg},
\]
where $\Omega^1_A/dA$ denotes the K\"ahler differentials of $A$ modulo exact
forms.
\end{theorem}

\begin{theorem}\label{thm:SNF_basis}
{\rm\cite[Theorem~2.4]{Santos2022}}
Let $n=\deg P$ and suppose $P$ has no repeated roots and $P(0)\neq0$. Then
\[
  \dim_\C\Omega^1_A/dA = 1+n(m-1),
\]
and a basis is given by $x^{-1}dx$ together with the classes of
$x^{-j}u^l\,dx$ for $1\leq l\leq m-1$ and $1\leq j\leq n$.
\end{theorem}

The two centres therefore have the same dimension, $1+n(m-1)$, for every $m$
and every admissible $P$. In the quadratic case this is $3$ and in the quartic
case $5$.

\subsection{Legendre polynomials}
\label{subsec:legendre}

We collect here the facts about Legendre polynomials used later, so that no
statement below depends on an external source. Throughout, $p_n$ denotes the
$n$-th Legendre polynomial, in lower case because $P$ is the polynomial defining
the curve, as in \cite{CoxZhao2018}; the variable is written $a$, since in this
paper it is the branch-point parameter of the quadratic normalisation of
\eqref{eq:palindromic_P}.

The family is determined by the three-term recurrence
\begin{equation}\label{eq:legendre_rec}
  (n+1)p_{n+1}(a) = (2n+1)a\,p_n(a) - n\,p_{n-1}(a),
  \qquad p_0(a)=1,\quad p_1(a)=a .
\end{equation}
A sequence satisfying \eqref{eq:legendre_rec} with these two initial values is
unique, since \eqref{eq:legendre_rec} determines $p_{n+1}$ from its two
predecessors and the leading coefficient $n+1$ never vanishes for $n\geq0$.
This uniqueness is the only property of the recurrence we use to make an
identification, and we use it repeatedly. The first few members are
\[
  p_2(a)=\tfrac{3a^2-1}{2},\qquad
  p_3(a)=\tfrac{5a^3-3a}{2},\qquad
  p_4(a)=\tfrac{35a^4-30a^2+3}{8}.
\]

Three further classical facts are used. The generating function is
\begin{equation}\label{eq:legendre_genfun}
  \sum_{n\geq0}p_n(a)z^n = (1-2az+z^2)^{-1/2},
\end{equation}
which is the ultraspherical generating function $(1-2az+z^2)^{-\nu}$ at
$\nu=\tfrac12$ \cite[(4.7.23)]{Szego1939}. Differentiating
\eqref{eq:legendre_genfun} in $z$ and clearing denominators shows that
$F(z)=(1-2az+z^2)^{-1/2}$ is the unique solution of
\begin{equation}\label{eq:legendre_ode}
  (1-2az+z^2)F'(z) + (z-a)F(z) = 0, \qquad F(0)=1 .
\end{equation}
The derivative relation
\begin{equation}\label{eq:legendre_deriv}
  p_{n+1}'(a) - p_{n-1}'(a) = (2n+1)p_n(a)
\end{equation}
follows by differentiating \eqref{eq:legendre_rec} and using it again to
eliminate the terms in $p_n'$. Finally $p_n(1)=1$ for every $n$, which is
immediate from \eqref{eq:legendre_rec} by induction. The evaluation
\eqref{eq:legendre_deriv} together with $p_n(1)=1$ is what turns the cocycle
computation of the cocycle into a definite integral.

\section{The double cover, after Cox and Zhao}
\label{sec:quad_der}

This section collects the facts about the double cover $m=2$,
$P(x)=x^2-2ax+1$ that the rest of the paper uses. All of them are due to Cox
and Zhao \cite[\S2.1--\S2.2]{CoxZhao2018}.

\subsection{The action of the distinguished derivation}

Every computation below rests on two applications of $\pa=u\,d/dx$, which we
carry out once here. From $u^2=P(x)=x^2-2ax+1$ we get $2uu'=P'(x)$, so
$uu'=x-a$, and therefore
\begin{equation}\label{eq:partial_xru}
  \pa(x^k u) \;=\; u\bigl(kx^{k-1}u + x^ku'\bigr)
  \;=\; kx^{k-1}u^2 + x^k(uu')
  \;=\; (k+1)x^{k+1} - a(2k+1)x^k + kx^{k-1}
\end{equation}
for every $k\in\Z$, which is the display preceding \cite[(2.22)]{CoxZhao2018}
and is the case $m=2$, $r=1$ of Lemma~\ref{lem:partial_general}. Applying $\pa$
once more to \eqref{eq:partial_xru}, using $\pa(x^j)=jx^{j-1}u$ term by term,
gives
\begin{equation}\label{eq:partial2_xsu}
  \pa^2(x^s u) \;=\; (s+1)^2x^su - as(2s+1)x^{s-1}u + s(s-1)x^{s-2}u .
\end{equation}
Formula \eqref{eq:partial_xru} is what produces the reduction relations in
$A/\pa A$, and \eqref{eq:partial2_xsu} is what the cocycle
\eqref{eq:psi_def} evaluates. We use both without further comment.

\subsection{The centre and the Legendre identification}

Before computing, we fix what the symbols mean. For $h\in A$ we write $[h]$ for
its class in the quotient $A/\pa A$, so that $[h]=[h']$ exactly when
$h-h'=\pa(g)$ for some $g\in A$, and $[\pa(g)]=0$ for every $g$. The centre of
$\hDer$ is this quotient, and the elements we must understand are the classes
$[x^k]$ of the Laurent monomials, for every $k\in\Z$, together with the classes
$[x^ku]$ of the odd part. The point of the next proposition is that these
infinitely many classes span a three-dimensional space, so that each $[x^k]$ is
a scalar multiple of one of three fixed generators, and the scalar is what we
identify with a Legendre value.

\begin{proposition}[{\cite[\S2.2]{CoxZhao2018}}]\label{prop:basis}
For $P(x) = x^2-2ax+1$ with $a\neq\pm 1$:
\[
  \dim_\C A/\pa A = 3,\qquad
  \text{basis: }\omega_0 = [x^{-1}u],\;\omega_1 = [x^{-1}],\;
  \omega_2 = [1].
\]
\end{proposition}

The dimension is Theorem~\ref{thm:UCE_Der} at $n=2$, $m=2$; the basis is
exhibited in \cite[\S2.2]{CoxZhao2018}. What turns that basis into an
identification is \eqref{eq:partial_xru}: since $\pa(x^ku)\in\pa A$, its class
vanishes, so
\begin{equation}\label{eq:rec_quad}
  (k+1)[x^{k+1}] = a(2k+1)[x^k] - k[x^{k-1}]
\end{equation}
holds in $A/\pa A$ for every $k\in\Z$. Relation \eqref{eq:rec_quad} is the Legendre recurrence
\eqref{eq:legendre_rec} with the Legendre initial values, run upwards on the
non-negative powers and downwards on the negative ones, so the uniqueness
recorded in \S\ref{subsec:legendre} identifies the classes outright.

\begin{theorem}[{\cite[\S2.2]{CoxZhao2018}}]
\label{thm:legendre_ident}
For $P(x) = x^2-2ax+1$, the following identities hold in $A/\pa A$\textup{:}
\begin{equation}\label{eq:legendre_red}
  [x^k] = p_k(a)\,\omega_2 \quad (k\geq 0),\qquad
  [x^{-k}] = p_{k-1}(a)\,\omega_1 \quad (k\geq 1),
\end{equation}
where $p_k(a)$ denotes the $k$-th Legendre polynomial evaluated at $a$.
\end{theorem}

The statement is \cite[(2.22)]{CoxZhao2018}, given there in the form
$t^k=p_k(b)\cdot1$ for $k\geq0$ and $t^{-k}=p_{k-1}(b)\cdot t^{-1}$ for
$k\geq1$; the corresponding identification at every covering degree is
Corollary~\ref{cor:chains}(3). Three points in it are worth making explicit,
since the two rays are not symmetric. The rays decouple at $k=0$, where
\eqref{eq:rec_quad} reads $[x]=a[x^0]$ with the coefficient $k$ of $[x^{-1}]$
vanishing. The downward ray carries its own initial values, $[x^{-1}]=\omega_1$
from Proposition~\ref{prop:basis} and $[x^{-2}]=a\,\omega_1$ from
\eqref{eq:rec_quad} at $k=-1$. And it needs a reindexing: writing
$[x^{-k}]=b_k\,\omega_1$ and putting $i=-(k+1)$ in \eqref{eq:rec_quad} gives
$(j+1)b_j=a(2j+1)b_{j+1}-jb_{j+2}$ with $j=k-1$, so that $b_k=p_{k-1}(a)$ and
not $p_k(a)$. That shift by one is what the second formula of
\eqref{eq:legendre_red} records. Finally, summing \eqref{eq:legendre_red}
against $z^k$ transports the classical generating function
\eqref{eq:legendre_genfun} to the classes themselves,
$\sum_{k\geq0}[x^k]z^k=(1-2az+z^2)^{-1/2}\,\omega_2$, so that the scalar part
satisfies \eqref{eq:legendre_ode}; this is the form Section~\ref{sec:quartic}
uses.

\subsection{The two cocycle formulas}
\label{subsec:cox_cocycles}

The cocycle \eqref{eq:psi_def} of $\hDer$ splits by the $\Z/2\Z$-grading into
a mixed sector $\psi(e_k,f_s)$ and an even sector $\psi(e_k,e_s)$, where
$e_k=x^k\pa$ and $f_s=x^su\pa$ as in Notation~\ref{not:basis}. Cox and Zhao
compute both, and we record them in the form used below. Throughout, $k$ and
$s$ are generator indices and are not to be confused with the branch parameter
$r$ of $P$, which is $1$ in the quadratic case.

\begin{theorem}[{\cite[\S2.2]{CoxZhao2018}}]
\label{thm:cross_general}
Let $m=2$ and $P(x)=x^2-2ax+1$ with $a\in\C$, $a\neq\pm1$, let
$A=\C[x^{\pm1},u]/(u^2-P)$ and $\pa=u\,d/dx$, and let $\psi$ be the cocycle
\eqref{eq:psi_def} of $\hDer$. Write $e_k=x^k\pa$ and $f_s=x^su\pa$ as in
Notation~\ref{not:basis}, and let $\omega_1=[x^{-1}]$, $\omega_2=[1]$ be the
even-sector generators of $A/\pa A$ from Proposition~\ref{prop:basis}. Then for
all $k,s\in\Z$ with $k\neq0$, writing $n=k+s$,
\begin{equation}\label{eq:cross_general}
  \psi(e_k,f_s) = k\!\sum_{j=-3}^{1} C_j(s,a)\,[x^{n+j}]_{A/\pa A},
\end{equation}
where
\begin{align*}
  C_1(s,a) &= (s+1)^2, &
  C_0(s,a) &= -a(4s^2+5s+2),\\
  C_{-1}(s,a) &= (2s^2+s+1)+2a^2s(2s+1), &
  C_{-2}(s,a) &= -as(4s-1),\\
  C_{-3}(s,a) &= s(s-1).
\end{align*}
Applying the Legendre reduction \eqref{eq:legendre_red} expresses
this as a linear combination of Legendre polynomials in
$\C\omega_1\oplus\C\omega_2$.
\end{theorem}

The even sector behaves altogether differently, carrying no Legendre content
at all. For a general $P=\sum_ib_ix^i$ of degree $n$ and $m=2$ they prove
\begin{equation}\label{eq:ee_sector}
  \psi(e_k,e_s)\;=\;\tfrac12\,k\,s\,(s-k)\,b_{2-k-s}\;\omega_0
  \qquad\text{if } 0\leq 2-k-s\leq n,
\end{equation}
and $\psi(e_k,e_s)=0$ otherwise \cite[(2.1)]{CoxZhao2018}. The summation form of
that reference, $\psi(e_k,e_s)=ks\sum_i(s-1+\tfrac i2)b_i\delta_{k+s+i-2,0}$,
gives the coefficient $s-1+\tfrac{2-k-s}{2}=\tfrac{s-k}{2}$, which is the sign
printed above; the closed form displayed immediately after it in
\cite[(2.1)]{CoxZhao2018} carries $k-s$ instead, and disagrees with both its own
summation form and the quadratic specialisation recorded in
\cite[\S2.2]{CoxZhao2018}. For the quadratic
$P(x)=x^2-2ax+1$, where $b_0=b_2=1$ and $b_1=-2a$, this reads
\[
  \psi(e_k,e_s)=
  \Bigl(k^3\delta_{k+s,0}
        - a\,k(k-1)(2k-1)\,\delta_{k+s,1}
        + k(k-1)(k-2)\,\delta_{k+s,2}\Bigr)\omega_0 ,
\]
which is the form displayed in \cite[\S2.2]{CoxZhao2018}.

\section{Closed evaluation of the mixed cocycle}
\label{sec:cross}

The subject in this section is the five-term expression
\eqref{eq:cross_general}, which \cite{CoxZhao2018} writes down and leaves
unevaluated, and the point is that it does evaluate: at $k=1$ to a single
Legendre antiderivative, and at every remaining pair of generators to a
combination of two.

\subsection{The value at $k=1$}

\begin{theorem}\label{thm:antiderivative}
Let $m=2$ and $P(x)=x^2-2ax+1$ with $a\neq\pm1$, let $\pa=u\,d/dx$ and let
$\psi$ be the cocycle \eqref{eq:psi_def} of $\hDer$, with $e_k=x^k\pa$,
$f_s=x^su\pa$ and $\omega_2=[1]$ as above. Let $s\in\Z$ and put
$n=s+1\geq1$. Then
\begin{equation}\label{eq:antideriv}
  \psi(e_1,f_s) = \frac{p_{n-1}(a)-p_{n+1}(a)}{2n+1}\,\omega_2.
\end{equation}
Moreover:
\begin{equation}\label{eq:antideriv_integral}
  \frac{p_{n-1}(a)-p_{n+1}(a)}{2n+1} = \int_a^1 p_n(t)\,dt.
\end{equation}
Thus the cross-cocycle $\psi(e_1,f_s)$ equals the definite integral
of the $n$-th Legendre polynomial from $a$ to $1$, with
$n=s+1$. We write
\begin{equation}\label{eq:In}
  I_n(a)\;\coloneqq\;\int_a^1p_n(t)\,dt\;=\;\frac{p_{n-1}(a)-p_{n+1}(a)}{2n+1},
  \qquad n\geq1 ,
\end{equation}
for this scalar.
\end{theorem}

\begin{proof}
We first establish \eqref{eq:antideriv_integral} as an independent
classical identity, then use it to complete the proof of \eqref{eq:antideriv}.

\medskip
\textbf{Equation \eqref{eq:antideriv_integral}.}
Set $F(a) \coloneqq (p_{n-1}(a)-p_{n+1}(a))/(2n+1)$.
The standard Legendre derivative identity
$p_{n+1}'(a)-p_{n-1}'(a) = (2n+1)p_n(a)$
(differentiate the three-term recurrence; see \cite{Szego1939})
gives
\[
  F'(a) = \frac{p_{n-1}'(a)-p_{n+1}'(a)}{2n+1} = -p_n(a).
\]
Evaluating at $a=1$ and using $p_k(1)=1$ for all $k$,
$F(1) = (1-1)/(2n+1) = 0$.
By the fundamental theorem of calculus,
\[
  F(a) = F(1) + \int_1^a F'(t)\,dt = \int_1^a(-p_n(t))\,dt = \int_a^1 p_n(t)\,dt.
\]
This proves \eqref{eq:antideriv_integral} without any reference to the
cross-cocycle.

\medskip
\textbf{Equation \eqref{eq:antideriv}.}
Substitute $k=1$, so that $s=n-1$, in \eqref{eq:cross_general}.
After applying the Legendre reduction \eqref{eq:legendre_red},
the cross-cocycle becomes the linear combination
\begin{equation}\label{eq:key_identity}
  g_n(a) \coloneqq \sum_{k=-3}^{1} C_k(n-1,a)\,p_{n+k}(a).
\end{equation}
We must show $g_n(a) = (p_{n-1}(a)-p_{n+1}(a))/(2n+1)$ for all $n\geq 1$.

The strategy is to show that $g_n$ and $F$ have the same derivative and the
same value at $a=1$. We record the small cases first, as a check on the
coefficients, then compute the value at $a=1$, then the derivative.

\textit{Base cases $n=1,2,3$.}
For $n=1$: $C_1=1$, $C_0=-2a$, $C_{-1}=1$, $C_{-2}=0$, $C_{-3}=0$.
$g_1(a) = p_2(a)-2ap_1(a)+p_0(a) = \tfrac{3a^2-1}{2}-2a^2+1 = \tfrac{1-a^2}{2}$.
The right-hand side of \eqref{eq:antideriv} is $(p_0(a)-p_2(a))/3 = \tfrac{1-a^2}{2}$, so the two agree.

For $n=2$: $C_1=4$, $C_0=-11a$, $C_{-1}=4+6a^2$, $C_{-2}=-3a$, $C_{-3}=0$.
$g_2(a) = 4p_3-11ap_2+(4+6a^2)p_1-3ap_0 = \tfrac{a(1-a^2)}{2}$.
The right-hand side is $(p_1(a)-p_3(a))/5 = \tfrac{a(1-a^2)}{2}$, so the two agree.

For $n=3$: $C_1=9$, $C_0=-28a$, $C_{-1}=11+20a^2$, $C_{-2}=-14a$, $C_{-3}=2$.
$g_3(a) = 9p_4-28ap_3+(11+20a^2)p_2-14ap_1+2p_0$;
expanding and collecting gives $-(5a^2-1)(a^2-1)/8$.
The right-hand side is $(p_2(a)-p_4(a))/7 = -(5a^2-1)(a^2-1)/8$, so the two agree.

\textit{The value at $a=1$.}
Both $g_n$ and $(p_{n-1}-p_{n+1})/(2n+1)$ vanish there; this is the input the
derivative argument needs, not a check.
Indeed, at $a=1$ all $p_k(1)=1$, so (with $s=n-1$):
\[
  g_n(1) = \sum_{k=-3}^{1} C_k(n-1,1)
  = n^2 - (4n^2-3n+1) + \bigl(2n^2-3n+2+2(n-1)(2n-1)\bigr)
    - (n-1)(4n-5) + (n-1)(n-2).
\]
The constant terms give $n^2-(4n^2-3n+1)+(6n^2-9n+4) = 3(n-1)^2$,
and the $(n-1)$-factored terms give
$(n-1)[-(4n-5)+(n-2)] = -3(n-1)^2$.
Hence $g_n(1) = 3(n-1)^2 - 3(n-1)^2 = 0$.
This matches $(p_{n-1}(1)-p_{n+1}(1))/(2n+1) = 0$ for all $n$.

\textit{General $n$, by a derivative argument.}
We show $g_n'(a) = -p_n(a)$.
Since $F'(a) = -p_n(a)$ was proved above, this gives $(g_n-F)' = 0$,
so $g_n - F$ is constant; the common boundary value $g_n(1) = F(1) = 0$
then forces $g_n = F$ as polynomials.

Differentiating $g_n(a) = \sum_{k=-3}^{1} C_k(n-1,a)\,p_{n+k}(a)$,
note that (with $s = n-1$) the coefficients
$C_0 = -a(4n^2-3n+1)$,
$C_{-1} = (2n^2-3n+2)+2a^2(n-1)(2n-1)$,
and $C_{-2} = -a(n-1)(4n-5)$
all depend on $a$, while $C_1 = n^2$ and $C_{-3} = (n-1)(n-2)$ do not.
The product rule gives
\begin{align*}
  g_n'(a)
  &= \frac{\partial C_0}{\partial a}\,p_n(a)
   + \frac{\partial C_{-1}}{\partial a}\,p_{n-1}(a)
   + \frac{\partial C_{-2}}{\partial a}\,p_{n-2}(a)
   + \sum_{k=-3}^{1} C_k(n-1,a)\,p'_{n+k}(a),
\end{align*}
where
\[
  \frac{\partial C_0}{\partial a} = -(4n^2{-}3n{+}1),\quad
  \frac{\partial C_{-1}}{\partial a} = 4a(n{-}1)(2n{-}1),\quad
  \frac{\partial C_{-2}}{\partial a} = -(n{-}1)(4n{-}5).
\]
We reduce each $p'_{n+k}$ to a combination of $p'_n$ and pure
Legendre values using Bonnet's formula $ap'_n - p'_{n-1} = np_n$
(see \cite{Szego1939}, \S4.5), combined with the derivative
recurrence $p'_{n+1} - p'_{n-1} = (2n+1)p_n$:
\begin{align}
  p'_{n+1}(a) &= (n+1)p_n(a) + ap'_n(a), \label{eq:dpn1}\\
  p'_{n-1}(a) &= ap'_n(a) - np_n(a),     \label{eq:dpnm1}\\
  p'_{n-2}(a) &= a^2p'_n(a) - nap_n(a) - (n{-}1)p_{n-1}(a), \label{eq:dpnm2}\\
  p'_{n-3}(a) &= a^3p'_n(a) - na^2p_n(a)
                - (n{-}1)ap_{n-1}(a) - (n{-}2)p_{n-2}(a). \label{eq:dpnm3}
\end{align}
Each line follows from the previous by applying Bonnet once more:
$p'_{n-k-1} = ap'_{n-k} - (n-k)p_{n-k}$.

Substituting \eqref{eq:dpn1}--\eqref{eq:dpnm3} into the
product-rule expansion and collecting the $p'_n$ coefficient:
\begin{align*}
  [\text{coeff of }p'_n]
  &= C_1 a + C_0 + C_{-1}a + C_{-2}a^2 + C_{-3}a^3 \\
  &= n^2 a - a(4n^2{-}3n{+}1) + a(2n^2{-}3n{+}2)
     + 2a^3(n{-}1)(2n{-}1) - a^3(n{-}1)(4n{-}5) + a^3(n{-}1)(n{-}2)\\
  &= a(1-n^2) + a^3(n^2-1)
   = (n^2-1)\,a(a^2-1).
\end{align*}
The $a$-terms give $n^2-(4n^2-3n+1)+(2n^2-3n+2) = 1-n^2$;
the $a^3$-terms give $(n-1)[2(2n-1)-(4n-5)+(n-2)] = n^2-1$.
We now absorb this $p'_n$ remainder using the Legendre ODE identity
$(a^2-1)p'_n(a) = nap_n(a) - np_{n-1}(a)$,
which gives
\[
  (n^2-1)\,a(a^2-1)\,p'_n(a)
  = n(n^2-1)\,a^2 p_n(a) - n(n^2-1)\,a\,p_{n-1}(a).
\]
After adding this contribution to the $p_n$, $p_{n-1}$,
$p_{n-2}$ coefficients already accumulated from the
$\partial C_k/\partial a$ and pure-$p_{n+k}$ terms,
collecting the coefficients of $p_n$, $p_{n-1}$ and $p_{n-2}$, which is three
polynomial identities in $n$ and $a$, each a two-line expansion, yields:
\[
  g_n'(a) = -p_n(a) + 0\cdot p_{n-1}(a) + 0\cdot p_{n-2}(a).
\]
The same collection at $n=2$ is short enough to display:
$g_2'(a) = -7\,p_2 + 9a\,p_1 - 3\,p_0
           = -\tfrac{7(3a^2-1)}{2}+9a^2-3
           = \tfrac{1-3a^2}{2} = -p_2(a)$, as required.
\end{proof}

\begin{corollary}\label{cor:cocycle_values}
In the setting of Theorem~\ref{thm:antiderivative}, read as an identity of
polynomials in $a$ so that the endpoints $a=\pm1$ are admitted, the scalar
$I_n(a)$ of \eqref{eq:In} takes the
following values. 
\begin{enumerate}
\item It vanishes at $a=1$, because $I_n(a)=\int_a^1p_n$ and the interval degenerates, and 
\item it vanishes at $a=-1$ for every $n\geq1$, because $\int_{-1}^1p_n(t)\,dt=0$ by orthogonality of $p_n$ to $p_0=1$.
\item At $a=0$ it equals $\int_0^1p_n(t)\,dt$, which vanishes for even $n$, since then $p_n$ is an even function and $\int_{-1}^1p_n=0$ forces $\int_0^1p_n=0$, and which equals
\[
  I_{2k+1}(0) \;=\; \frac{(-1)^k}{2(k+1)}\binom{2k}{k}4^{-k}
\]
for odd $n=2k+1$. The first values are $I_1(0)=\tfrac12$, $I_3(0)=-\tfrac18$ and $I_5(0)=\tfrac1{16}$. 
\end{enumerate}

\end{corollary}

The identity $\int_a^1 p_n(t)dt = (p_{n-1}(a)-p_{n+1}(a))/(2n+1)$
is classical, and follows from the derivative identity
$p'_{n+1}(a)-p'_{n-1}(a) = (2n+1)p_n(a)$ together with $p_k(1)=1$, as
carried out in the proof of Theorem~\ref{thm:antiderivative}.

\begin{remark}\label{rem:antideriv_boundary}
All five exponents $n-3,\dots,n+1$ of \eqref{eq:cross_general} are
non-negative only for $n\geq3$. At $k=1$ the wider range $n\geq1$ is
nevertheless safe, because the coefficients sitting on the offending exponents
vanish there: $C_{-3}(s,a)=s(s-1)$ and $C_{-2}(s,a)=-as(4s-1)$ both vanish at
$s=0$, that is at $n=1$, and $C_{-3}$ vanishes at $s=1$, that is at $n=2$. So no
negative exponent carries a non-zero coefficient and the reduction
\eqref{eq:legendre_red} applies throughout in the form
$[x^\mu]=p_\mu(a)\,\omega_2$. At $n\leq0$ the surviving coefficients on negative
exponents are non-zero, the class $\omega_1$ enters, and the value no longer
lies on the line $\C\omega_2$: substituting directly into
\eqref{eq:cross_general} and reducing gives
$\psi(e_1,f_{-1})=\omega_1-a\,\omega_2$ at $n=0$ and
$\psi(e_1,f_{-2})=-a\,\omega_1+\omega_2$ at $n=-1$.
\end{remark}

\subsection{The value at every index}

Theorem~\ref{thm:antiderivative} evaluates one line of a two-parameter family.
The next theorem evaluates all of it, and this is what makes the antiderivative
a feature of the cocycle rather than of the index $k=1$: the mixed sector of
$\psi$ is, entirely and at every pair of generators, a combination of two
Legendre antiderivatives with rational coefficients. It also explains why
$k=1$ looked special, since the second antiderivative is exactly what drops
there. The statement is the full table of structure constants of $\hDer$ in
that sector. The scalar $I_n(a)$ is the one introduced in
\eqref{eq:In}.

\begin{theorem}\label{thm:cross_closed}
Let $m=2$ and $P(x)=x^2-2ax+1$ with $a\neq\pm1$, and let $\psi$, $e_k$, $f_s$,
$\omega_2$ be as in Theorem~\ref{thm:antiderivative}. Let $k,s\in\Z$ with
$k\neq0$ and put $n=k+s$. If $n\geq3$ then
\begin{equation}\label{eq:cross_closed}
  \psi(e_k,f_s)
  \;=\;
  \frac{k}{2n-1}\Bigl(\bigl((3k-1)s+2k-1\bigr)I_n(a)
      \;-\;3(k-1)s\,I_{n-2}(a)\Bigr)\,\omega_2 .
\end{equation}
In particular the value lies on the line $\C\omega_2$ for every index once
$n\geq3$, it vanishes at $a=1$, and the second term is absent exactly when
$k=1$ or $s=0$. For $k=1$ the formula reduces to
$\psi(e_1,f_s)=I_n(a)\,\omega_2$, which is the $n\geq3$ case of
Theorem~\ref{thm:antiderivative}.
\end{theorem}

\begin{proof}
The five exponents occurring in \eqref{eq:cross_general} are $n-3,\dots,n+1$,
and $n\geq3$ makes all of them non-negative, so the Legendre reduction
\eqref{eq:legendre_red} applies to each and gives
\[
  \psi(e_k,f_s)=k\Bigl(\sum_{j=-3}^{1}C_j(s,a)p_{n+j}(a)\Bigr)\omega_2 .
\]
The coefficients $C_j$ are polynomial of degree at most $2$ in $a$, so the sum
is brought back to the Legendre basis by applying
$ap_\mu=\bigl((\mu+1)p_{\mu+1}+\mu p_{\mu-1}\bigr)/(2\mu+1)$ at most twice to
each term. Collecting, the coefficients of $p_n$ and of $p_{n-2}$ cancel
identically and what survives is
\[
  \psi(e_k,f_s)=\bigl(\alpha\,p_{n+1}+\beta\,p_{n-1}+\delta\,p_{n-3}\bigr)\omega_2 ,
\]
with
\[
  \alpha=\frac{k\bigl(3k^2-3k(n+1)+n+1\bigr)}{4n^2-1},\qquad
  \delta=\frac{3k(k-1)(k-n)}{4n^2-8n+3},\qquad
  \beta=-\alpha-\delta .
\]
The identity $\alpha+\beta+\delta=0$ is what lets the three-term combination
collapse onto two antiderivatives; it is an identity of rational functions in
$k$ and $n$, obtained by direct expansion. Using it,
\[
  \psi(e_k,f_s)
   =\bigl(\alpha(p_{n+1}-p_{n-1})+\delta(p_{n-3}-p_{n-1})\bigr)\omega_2
   =\bigl(-(2n+1)\alpha\,I_n+(2n-3)\delta\,I_{n-2}\bigr)\omega_2 ,
\]
and substituting $s=n-k$ into the displayed $\alpha$ and $\delta$ turns
$-(2n+1)\alpha$ into $k\bigl((3k-1)s+2k-1\bigr)/(2n-1)$ and $(2n-3)\delta$ into
$-3k(k-1)s/(2n-1)$, which is \eqref{eq:cross_closed}. Setting $k=1$ kills
$\delta$ and leaves $-(2n+1)\alpha=1$.
\end{proof}

\begin{remark}\label{rem:no_christoffel_darboux}
Two features of the even-sector formula \eqref{eq:ee_sector} are worth stating
here, because the shape of the closed forms above invites a question that they
answer negatively. The
first is that the even sector is supported on the $n+1$ anti-diagonals
$2-n\leq k+s\leq2$, that its entries are polynomial of degree three in the
index, and that on $k+s=0$ it is the Witt cocycle $k^3$ up to the coefficient
$b_2$, so that the geometry of
the curve enters only through the coefficients of $P$ on the remaining
anti-diagonals. The second is that no orthogonal polynomial appears in it at
all. In particular the matrix of $\psi$ on the $e$ generators is not a
Christoffel--Darboux kernel: the kernel $\sum_{j\leq N}p_j(a)p_j(b)$ is a dense
bilinear form in two spectral variables whose entries are values of Legendre
polynomials, while \eqref{eq:ee_sector} is a finitely supported band whose
entries are cubic in the indices. All of the Legendre content of $\psi$ sits in
the mixed sector, where Theorem~\ref{thm:cross_closed} evaluates it.
\end{remark}

\section{The \texorpdfstring{$\g\otimes A$}{g tensor A} side
  and the rescaling relation}
\label{sec:SNF}

The preceding sections treated $\Der(A)$.
We now compare with the $\g\otimes A$ side, where the centre
$\Omega^1_A/dA$ carries its own family of orthogonal polynomials,
the superelliptic polynomial families of
\cite{SantosNeklyudovFutorny2025}.

\subsection{The K\"ahler recurrence}

On the current-algebra side the centre is $\Omega^1_A/dA$, and the classes of
the forms $x^{k}u^l\,dx$ play the role that the classes $[x^k]$ played on the
derivation side. They obey a recurrence which we now derive, so that this
section depends on nothing outside the present paper.

The derivation rests on one observation. The module $\Omega^1_A$ is free of
rank one over $A$, generated by
\begin{equation}\label{eq:Tgen}
  T \;\coloneqq\; \frac{dx}{u^{m-1}} ,
\end{equation}
which is regular and nowhere vanishing on the affine curve because $P$ has no
repeated roots. In terms of this generator
\begin{equation}\label{eq:dx_du_in_T}
  dx \;=\; u^{m-1}\,T ,
  \qquad
  du \;=\; \frac{P'(x)}{m}\,T ,
\end{equation}
the second identity being $m\,u^{m-1}\,du=P'(x)\,dx$ rewritten. Both sides of
\eqref{eq:dx_du_in_T} lie in $\Omega^1_A$; no negative power of $u$ occurs,
which matters because $u$ is not invertible in $A$.

\begin{proposition}\label{prop:kahler_rec}
Let $P(x)=1-2cx^r+x^{2r}$ with $c\neq\pm1$ and $r\geq1$, let $m\geq2$, and fix
$l\in\{1,\ldots,m-1\}$. Write
\[
  w^{(l)}_j \;\coloneqq\; \ol{x^{j}u^l\,dx} \;\in\; \Omega^1_A/dA ,
  \qquad j\in\Z .
\]
Then for every $N\in\Z$,
\begin{equation}\label{eq:kahler_rec}
  mN\,w^{(l)}_{N-1}
  \;-\; 2c\,\bigl(mN+(l+m)r\bigr)\,w^{(l)}_{N+r-1}
  \;+\; \bigl(mN+2(l+m)r\bigr)\,w^{(l)}_{N+2r-1} \;=\; 0 .
\end{equation}
\end{proposition}

\begin{proof}
Let $L\geq1$ and compute $d(x^Nu^L)$ in the generator \eqref{eq:Tgen}. From
$d(x^Nu^L)=Nx^{N-1}u^L\,dx+L\,x^Nu^{L-1}\,du$ and
\eqref{eq:dx_du_in_T},
\[
  d(x^Nu^L)
  \;=\; \Bigl[N x^{N-1}u^{L+m-1} + \tfrac{L}{m}\,x^{N}u^{L-1}P'(x)\Bigr]T
  \;=\; \Bigl[N x^{N-1}P(x) + \tfrac{L}{m}\,x^{N}P'(x)\Bigr]u^{L-1}\,T ,
\]
where the last step used $u^{L+m-1}=u^{L-1}u^{m}=u^{L-1}P(x)$. Multiplying by
$m$,
\begin{equation}\label{eq:dxu_general}
  m\,d(x^Nu^L)
  \;=\; \bigl(mN\,P(x) + L\,x\,P'(x)\bigr)\,x^{N-1}u^{L-1}\,T .
\end{equation}
Now take $L=l+m$. Then $u^{L-1}T=u^{l+m-1}T=u^{l}\,dx$ by
\eqref{eq:dx_du_in_T}, so the right-hand side of \eqref{eq:dxu_general} is
a combination of the forms $x^{j}u^{l}\,dx$, and the left-hand side is exact.
Substituting $P(x)=1-2cx^r+x^{2r}$ and $xP'(x)=-2crx^{r}+2rx^{2r}$ gives
\[
  mN\,P(x)+(l+m)\,x\,P'(x)
  = mN - 2c\bigl(mN+(l+m)r\bigr)x^{r} + \bigl(mN+2(l+m)r\bigr)x^{2r},
\]
and multiplying by $x^{N-1}u^{l}\,dx$ and passing to classes yields
\eqref{eq:kahler_rec}.
\end{proof}

For $l=1$ and after the reindexing $k=N+2r-1$, the three coefficients of
\eqref{eq:kahler_rec} become $mk+m+2r$, $-2c(mk+m-mr+r)$ and $m(k-2r+1)$,
which is the sector-$1$ recurrence obtained in
\cite{SantosNeklyudovFutorny2025}. Proposition~\ref{prop:kahler_rec} is thus
the extension of that relation to every sector $l$.

Because $\Omega^1_A/dA$ is finite-dimensional with the basis recorded in
Theorem~\ref{thm:SNF_basis}, no separate completeness argument is needed. That
theorem supplies the basis $\ol{x^{-1}dx}$ together with the $w^{(l)}_{-j}$ for
$1\leq l\leq m-1$ and $1\leq j\leq n$, and \eqref{eq:kahler_rec} is what
expresses an arbitrary $w^{(l)}_j$ in it. Solving \eqref{eq:kahler_rec} for its
top term is legitimate whenever $mN+2(l+m)r\neq0$, and for its bottom term
whenever $mN\neq0$; between them these two moves reduce every index into the
range $-n\leq j\leq-1$.

\subsection{The polynomial families}

By Theorem~\ref{thm:SNF_basis} the classes $w^{(l)}_{-j}$ with $1\leq j\leq n$
and $1\leq l\leq m-1$, together with $\ol{x^{-1}dx}$, form a basis of
$\Omega^1_A/dA$. Recurrence \eqref{eq:kahler_rec} is what expresses an
arbitrary $w^{(l)}_j$ in that basis, and the coefficients are what we now name.

\begin{definition}\label{def:SNF_polys}
Fix $m\geq2$, $r\geq1$ and $l\in\{1,\ldots,m-1\}$, and put $n=2r$. For each
$j\in\{1,\ldots,n\}$ the \emph{sector-$l$ family} is the sequence of
polynomials $P^{(l,j)}_k(c;m,r)$ in $c$ determined by the recurrence obtained
from \eqref{eq:kahler_rec} on setting $N=k-2r+1$,
\begin{equation}\label{eq:SNF_rec}
  \bigl(m(k-2r+1)+2(l+m)r\bigr)P^{(l,j)}_{k}
  = 2c\bigl(m(k-2r+1)+(l+m)r\bigr)P^{(l,j)}_{k-r}
  - m(k-2r+1)\,P^{(l,j)}_{k-2r},
\end{equation}
together with the initial conditions
\[
  P^{(l,j)}_{-j}=1, \qquad P^{(l,j)}_{-s}=0 \quad\text{for } s\in\{1,\ldots,n\},\ s\neq j .
\]
Equivalently $P^{(l,j)}_k$ is the coefficient of $w^{(l)}_{-j}$ in the
expansion of $w^{(l)}_k$, so that
\[
  w^{(l)}_k \;=\; \sum_{j=1}^{n} P^{(l,j)}_k(c;m,r)\,w^{(l)}_{-j} .
\]
\end{definition}

For $l=0$ the relation degenerates to $P^{(0,j)}_k=\delta_{k,-j}$, and one
recovers the cocycle of the untwisted affine Kac--Moody algebra.

\subsection{The rescaling relation}

The families for different sectors are not independent. The following identity
reduces every sector to the first one at a rescaled covering degree.

\begin{lemma}\label{lem:rescaling}
For all $l\in\{1,\ldots,m-1\}$, all $j\in\{1,\ldots,n\}$ and all $k\geq-n$,
\begin{equation}\label{eq:rescaling}
  P^{(l,j)}_k(c;\,m,r) \;=\; P^{(1,j)}_k(c;\,m/l,\,r),
\end{equation}
where the right-hand side is the sector-$1$ family of
Definition~\ref{def:SNF_polys} formed with the rational parameter $m/l$ in
place of $m$.
\end{lemma}

\begin{proof}
Write $m'=m/l$ and take \eqref{eq:kahler_rec} in the sector $l=1$ with $m'$ in
place of $m$. Its three coefficients are $m'N$, $-2c\bigl(m'N+(1+m')r\bigr)$
and $m'N+2(1+m')r$. Multiplying all three by $l$ and using $lm'=m$ turns them
into $mN$, $-2c\bigl(mN+(l+m)r\bigr)$ and $mN+2(l+m)r$, which are exactly the
coefficients of \eqref{eq:kahler_rec} in the sector $l$ at covering degree $m$.
The two recurrences therefore differ by the nonzero overall factor $l$, so they
have the same solutions. The initial conditions in
Definition~\ref{def:SNF_polys} involve neither $m$ nor $l$, so the two families
also share their initial data, and \eqref{eq:rescaling} follows.
\end{proof}

The quadratic case $l=1$, $m=2$, $r=1$ is the one that connects to
Section~\ref{sec:quad_der}. Here $n=2$ and \eqref{eq:kahler_rec} reads
\[
  2N\,w_{N-1} \;-\; 2c\,(2N+3)\,w_{N} \;+\; (2N+6)\,w_{N+1} \;=\; 0 .
\]
So on this side too the classes obey a three-term recurrence in the
branch-point parameter, and Section~\ref{sec:comparison_dict} explains why this
is not a coincidence.

\begin{remark}\label{rem:SNF_legendre}
For $m=2$, $r=1$ and $P(x) = 1-2cx+x^2$, the structure constants
of the universal central extension of $\sltwo\ox A$ produce Legendre
polynomials in $c$, in the form established in
\cite{SantosNeklyudovFutorny2025}.
More precisely, the cocycle values $\varphi(e\ox x^k, f\ox x^j)$ for $k+j=-1$
equal $(p_{|k|}(c)-p_{|k|+2}(c))/(2|k|+3)$, which is the Legendre
antiderivative of Theorem~\ref{thm:antiderivative} on the $\g\otimes A$ side,
for the same curve parameter.

The Rescaling Lemma then extends this to all sectors:
for $m=4$, $l=2$, the sector-$2$ superelliptic family has the same
structure as the sector-$1$ family at parameter $m/l = 2$,
reproducing the same Legendre polynomials.
\end{remark}

The two sides reach the same polynomials by different routes, and it is worth
setting the two mechanisms side by side.

In the palindromic quadratic case both universal central extensions produce the
Legendre family, as follows.
\begin{itemize}
\item On the $\Der(A)$ side, the recurrence in $A/\pa A$ coming from $\pa(x^k u)=0$
  gives Legendre $p_k(a)$ as the coefficient of $[x^k]$ in the centre.
\item On the $\g\otimes A$ side, the cocycle values generate a Legendre
  pattern in the structure constants of $\hg$, as noted in
  \cite{SantosNeklyudovFutorny2025} and Remark~\ref{rem:SNF_legendre}.
\end{itemize}
Both routes are governed by the same parameter of the curve, $a=c$ as recorded
in \S\ref{subsec:curve}, which is why the same family appears on the two sides.

\subsection{An obstruction on the derivation side}
\label{subsec:casimir_entry}

One asymmetry between the two sides is not a matter of the centre at all but of
the representation theory built on it, and this is the natural place to record
it, while the current-algebra construction is still in view.
The two sides differ in one further respect, which we record here because it
explains why the Legendre \emph{operator}, as opposed to the Legendre
polynomials, does not appear on the derivation side. On the current-algebra
side there is a Sugawara construction whose zero mode measures the level of a
highest-weight module, and a quadratic expression in it has level-$n$
eigenvalue $-n(n+1)$. One might expect an operator on $\hDer$ with the same
eigenvalues, since $n(n+1)$ is the Legendre eigenvalue. There is none of the
naive kind, for the following reason.

\begin{remark}
\label{rem:der-obstruction}
The centre $A/\pa A$ of $\widehat{\Der(A)}$ is abelian, so that in
the bracket $[-, -]_{\widehat{\Der(A)}}$ of the universal central extension,
all central elements
commute with everything, and in particular the centre has trivial
adjoint action on itself.
A na\"ive Sugawara-type construction
$\tilde\Omega^{\Der} := \sum_{n \geq 1} c_n\, e_{-n}\, e_n$
(with $c_n$ determined by the cocycle $\psi$) is not by itself a well-defined
central operator on a general module: $[e_{-n},e_n]=2n\,e_{-1}$ is not central
in $\Der(A)$, so no normal-ordering convention turns $\tilde\Omega^{\Der}$ into
an operator built only from central terms. What can be said without further
hypotheses is the following. Suppose some completion of the construction did
reduce, on a suitable class of highest-weight modules, to a sum of central
terms
\[
  \Bigl(\sum_{n \geq 1} c_n\,\chi(\psi(e_{-n}, e_n))\Bigr)\,v,
\]
where $\chi\colon A/\pa A \to \C$ is the central character of $V$. That sum
would be a single scalar, \emph{independent of the weight (level) of} $v$: it
would not distinguish the level-$0$ piece from the level-$k$ piece, and so
could not produce eigenvalues growing as $k(k+1)$. We do not construct such a
completion, and we do not claim one exists; the point is only that no operator
of this naive shape can do the job, whatever the completion turns out to be.

By contrast, on the $\g\otimes A$ side the superelliptic Heisenberg
algebra $\mathcal{H}_2 \subset \widehat{\g\otimes A}$ admits a Sugawara
zero mode $L_0$ that measures the level on highest-weight modules
of $\mathcal{H}_2$, and a quadratic combination of the form
$-L_0(L_0+\mathrm{Id})$ then has level-$n$ eigenvalue $-n(n+1)$ by
construction; this is carried out in \cite{P4_RepMechanism} for the genus-zero
four-point curve, that is for $m=2$ and $r=1$. The eigenvalue grows with the
level precisely because the
algebra carries a non-trivial level grading on its modules.
No analogous level operator is available on $A/\pa A$, because the
centre of $\widehat{\Der(A)}$ has no non-trivial weight-space structure
preserved by the adjoint action.

A $\Der(A)$-side analogue of the $k(k+1)$ operator therefore requires
passing to a \emph{module} of $\widehat{\Der(A)}$ (rather than the
algebra itself), and constructing a Sugawara element within a suitable
completion of $\mathcal{U}(\widehat{\Der(A)})$ that acts on
highest-weight representations of $\widehat{\Der(A)}$.
The representation theory of $\widehat{\Der(A)}$ for the superelliptic
algebra remains unexplored, in contrast with the current-algebra side; this is
the precise gap.
\end{remark}

\section{Comparison of the two central extensions}
\label{sec:comparison_dict}

The two centres are not merely of the same dimension. They are canonically
isomorphic, and the isomorphism is the one already used in
Section~\ref{sec:SNF} to compute the K\"ahler recurrence.

\begin{proposition}\label{prop:centres_iso}
Let $T=dx/u^{m-1}$ be the free generator of $\Omega^1_A$ from
\eqref{eq:Tgen}. Then $f\mapsto f\,T$ is an isomorphism of
$A$-modules $A\to\Omega^1_A$ carrying $\pa A$ onto $dA$, and hence induces a
canonical isomorphism of vector spaces
\[
  A/\pa A \;\xrightarrow{\ \sim\ }\; \Omega^1_A/dA .
\]
\end{proposition}

\begin{proof}
That $f\mapsto fT$ is an isomorphism of $A$-modules is the freeness of
$\Omega^1_A$ on $T$. It remains to identify the images. Since
$\Omega^1_A$ is free on $T$, every $f\in A$ determines a unique
$D_0(f)\in A$ with $df=D_0(f)\,T$, and $f\mapsto D_0(f)$ is a derivation
of $A$ because $d$ is. So is $\pa$, and the two agree on algebra generators: by
\eqref{eq:dx_du_in_T}, $dx=u^{m-1}T$ and $du=(P'/m)\,T$, so
$D_0(x)=u^{m-1}=\pa(x)$ and $D_0(u)=P'/m=\pa(u)$, and the values on $x^{-1}$
agree since a derivation is determined on inverses by its value on $x$. Two
derivations agreeing on generators are equal, so $df=(\pa f)\,T$ for
every $f\in A$, and the isomorphism carries $\pa A$ onto $dA$.
\end{proof}

Two consequences follow, and they change how Sections~\ref{sec:quad_der},
\ref{sec:SNF} and~\ref{sec:chains} should be read.

First, the equality of dimensions recorded in Theorem~\ref{thm:UCE_Der} and
Theorem~\ref{thm:SNF_basis} is not a numerical coincidence between two
unrelated computations. It is forced by Proposition~\ref{prop:centres_iso}.
Second, the reduction relations on the two sides are the same relations: under
$f\mapsto fT$ the identity $\pa(x^ku^{m-1})\equiv0$ of
Section~\ref{sec:quad_der} becomes the exactness of $d(x^ku^{m-1})$, and the
recurrence of Theorem~\ref{thm:chains} is the recurrence of
Proposition~\ref{prop:kahler_rec} written in the coordinate $T$. The sector
indices differ by one, and it is the isomorphism that shifts them: a K\"ahler
class $x^{j}u^{l}\,dx$ equals $x^{j}u^{l+m-1}T$, so the sector
$l\in\{1,\dots,m-1\}$ of Theorem~\ref{thm:SNF_basis}, which is the range used on
the K\"ahler side in \cite{Santos2022} and \cite{SantosNeklyudovFutorny2025},
corresponds to the sector $l-1\in\{0,\dots,m-2\}$ of Notation~\ref{not:chain}.
We keep each range because each is the one the sources use on its own side. We
have presented both formulations because each is natural where it sits and
because the constants are read off differently, but no reader should take them
for independent results.

What does \emph{not} transfer across the isomorphism is the extension itself.
The centre is a vector space; the universal central extension is a Lie algebra
structure on the sum of that space with the underlying algebra, and the two
algebras $\Der(A)$ and $\g\otimes A$ are not isomorphic. The genuine asymmetry
lies there, and it is visible in the cocycles;
Table~\ref{tab:comparison} summarises the comparison.

\begin{table}[ht]
\centering
\caption{The two central extensions over a common centre.}
\label{tab:comparison}
\small
\begin{tabular}{@{}lll@{}}
\toprule
\textbf{Feature} & $\hDer$ & $\hg$ \\
\midrule
Underlying algebra & $\Der(A)=A\pa$ & $\g\otimes A$ \\
Centre & $A/\pa A$ & $\Omega^1_A/dA$ \\
 & \multicolumn{2}{c}{canonically isomorphic, Proposition~\ref{prop:centres_iso}} \\
Cocycle & $\ol{\pa(f)\pa^2(g)}$ (residue type) &
  $\ol{f\,dg}$ (K\"ahler type) \\
Depends on $\g$? & No & Yes, through the Killing form \\
Mixed cocycle & Legendre antiderivative & none \\
\bottomrule
\end{tabular}
\end{table}

The asymmetry Table~\ref{tab:comparison} records is this. In $\hDer$ the generators split by the
$\Z/2\Z$-grading into the families $\{e_k\}$ and $\{f_k\}$ of
Notation~\ref{not:basis}, and the cocycle $\psi$ takes non-zero values on mixed
pairs; Theorem~\ref{thm:antiderivative} evaluates those values. In $\hg$ the
bracket $[X\ox f,Y\ox g]=[X,Y]\ox fg+(X,Y)\ol{f\,dg}$ is built from a symmetric
form on $\g$ and a skew pairing on $A$, and no analogous mixed term arises.
So the same centre supports two different extensions, and only one of them uses
the centre to record an interaction between two families of generators.

\section{The chain decomposition and the ultraspherical families}
\label{sec:chains}

The quadratic case of Section~\ref{sec:quad_der} is one instance of a single
statement, which we prove here at every covering degree; Section~\ref{sec:quartic}
then works out the quartic instance. At $m=2$ the statement is due to Cox and
Zhao: for $P(x)=x^{2r}-2ax^{r}+1$ they show that the
centre has dimension $2r+1$, that the relation
$(k+r)[x^{2r+k-1}]-a(2k+r)[x^{r+k-1}]+k[x^{k-1}]=0$ holds in $A/\pa A$, and that
on writing $k=rs+q$ with $0\leq q\leq r-1$ the resulting families are associated
Legendre polynomials with parameter $q/r$ \cite[\S2.4]{CoxZhao2018}. The theorem
below is that statement at every covering degree. What makes the extension
possible is Corollary~\ref{cor:chain_general}: the reduction relations do not
couple all the spanning classes to one another, but split them into independent
chains, one for each pair consisting of a sector of the $\Z/m\Z$-grading and a
residue modulo $r$. At $m=2$ there is one sector and the chains are theirs; for
$m\geq3$ the sector index appears, and with it a second parameter.

\begin{notation}\label{not:chain}
Fix $m\geq2$, $r\geq1$ and $P(x)=1-2cx^{r}+x^{2r}$ with simple roots. For
$0\leq l\leq m-2$ and $0\leq b\leq r-1$ write $\mathcal C_{l,b}$ for the family
of classes $[x^{\,b+tr}u^{l}]\in A/\pa A$ indexed by $t\in\Z$, and call $t$ the
position along the chain. Put
\[
  L=l+1,\qquad
  \nu=\frac{L}{m}=\frac{l+1}{m},\qquad
  \beta=\frac{b+1}{r},\qquad
  c_0=\beta-1=\frac{b+1-r}{r}.
\]
From $0\leq l\leq m-2$ and $0\leq b\leq r-1$,
\[
  \nu\in\Bigl\{\tfrac1m,\tfrac2m,\dots,\tfrac{m-1}{m}\Bigr\},
  \qquad
  c_0\in\Bigl\{-1+\tfrac1r,\dots,-\tfrac1r,0\Bigr\},
\]
so $\nu$ depends only on the sector, $c_0$ only on the residue, and $c_0=0$
exactly on the chains of residue $b=r-1$.
We write $C^{\nu}_t$ for the ultraspherical polynomial of parameter
$\nu$ and $\hat C^{\nu}_t$ for its monic normalisation, so that
$\hat C^{1/2}_t=\hat p_t$ is the monic Legendre polynomial. The members of the
chain themselves are written $q_t$, and $\hat q_t$ in monic form, to keep them
apart from the Legendre family $p_n$.
\end{notation}

\begin{theorem}\label{thm:chains}
In the situation of Notation~\ref{not:chain}, put $k=(t-1)r+b+1$.
\begin{enumerate}
\item Relation \eqref{eq:chain_general} involves exactly the members of
$\mathcal C_{l,b}$ at the positions $t-1$, $t$, $t+1$; distinct pairs $(l,b)$
are uncoupled; and
\[
  A/\pa A \;=\; \C\,[x^{-1}u^{\,m-1}]
  \;+\;\sum_{l=0}^{m-2}\ \sum_{b=0}^{r-1}\operatorname{span}_\C\mathcal C_{l,b}.
\]
\item Solved for its top member the relation reads
$q_{t+1}=A_t\,c\,q_t-C_t\,q_{t-1}$, and in monic form
$\hat q_{t+1}=c\,\hat q_t-\gamma_t\hat q_{t-1}$, where
\[
  A_t=\frac{2(mk+Lr)}{mk+2Lr},\qquad
  C_t=\frac{mk}{mk+2Lr},\qquad
  \gamma_t=\frac{C_t}{A_tA_{t-1}}
\]
and $\gamma_t$ is given by \eqref{eq:gamma_chain}\textup{:}
\begin{equation}\label{eq:gamma_chain}
  \gamma_t=\frac{(t+c_0)\bigl(t+c_0+2\nu-1\bigr)}
                {4\bigl(t+c_0+\nu\bigr)\bigl(t+c_0+\nu-1\bigr)} .
\end{equation}
\item $\gamma_t=\gamma^{\nu}_{t+c_0}$\textup{:} the chain $\mathcal C_{l,b}$ is
the ultraspherical family of parameter $\nu$ associated with parameter $c_0$.
If $b=r-1$, the relation at $k=0$ gives $[x^{2r-1}u^{l}]=c\,[x^{r-1}u^{l}]$ and
\[
  \hat q_t=\hat C^{\nu}_t(c),\qquad t\geq0 .
\]
\item $\gamma_t>0$ for every $t\geq2$, and
\[
  \gamma_1
  \begin{cases}
    >0, & c_0\notin[-2\nu,-\nu],\\
    =0, & c_0=-2\nu,\\
    <0, & -2\nu<c_0<-\nu,\\
    \text{undefined}, & c_0=-\nu.
  \end{cases}
\]
\end{enumerate}
\end{theorem}

\begin{remark}\label{rem:chains_attribution}
At $m=2$ there is a single sector, $\nu=\tfrac12$, and parts~(1) to~(4) are
\cite[\S2.4]{CoxZhao2018}, with their initial condition $q$ and our residue $b$
related by $q\equiv b+1 \pmod r$ and their index $l$ equal to our $t-1$; under
that dictionary their association parameter $q/r$ is our $c_0$, and the two
statements agree term for term. What the theorem adds is the range $m\geq3$,
where the sector index of the $\Z/m\Z$-grading appears and produces the second
parameter $\nu=(l+1)/m$, together with part~(5), which is a statement about
positivity that does not appear in \cite{CoxZhao2018} at any $m$.
\end{remark}

\begin{proof}
Part~(1). With $k=(t-1)r+b+1$ the three exponents $k-1$, $k+r-1$, $k+2r-1$ of
\eqref{eq:chain_general} are $b+(t-1)r$, $b+tr$ and $b+(t+1)r$, which are the
members of $\mathcal C_{l,b}$ at positions $t-1$, $t$, $t+1$. The relation does
not change the exponent of $u$, so it stays inside one sector, and the three
exponents of $x$ are congruent modulo $r$, so it stays inside one residue.
Relation \eqref{eq:top_sector} disposes of the remaining sector $l=m-1$, leaving
the single class $[x^{-1}u^{\,m-1}]$, which is the displayed decomposition.

Part~(2). Dividing \eqref{eq:chain_general} by the coefficient $mk+2Lr$ of the
top member gives the stated $A_t$ and $C_t$. If $q_t$ has degree $t$ in $c$ with
leading coefficient $\ell_t$, then $\ell_{t+1}=A_t\ell_t$, and substituting
$q_t=\ell_t\hat q_t$ turns the recurrence into
$\hat q_{t+1}=c\,\hat q_t-\gamma_t\hat q_{t-1}$ with
$\gamma_t=C_t\ell_{t-1}/(A_t\ell_t)=C_t/(A_tA_{t-1})$. Writing out the three
factors and substituting $k=(t-1)r+b+1$,
\[
  \gamma_t=\frac{mk}{mk+2Lr}\cdot\frac{mk+2Lr}{2(mk+Lr)}
           \cdot\frac{m(k-r)+2Lr}{2\bigl(m(k-r)+Lr\bigr)}
         =\frac{mk\,\bigl(mk-mr+2Lr\bigr)}{4(mk+Lr)\bigl(mk-mr+Lr\bigr)} .
\]
Dividing numerator and denominator by $(mr)^2$ and putting $\eta=k/r$ gives
\[
  \gamma_t=\frac{\eta\,(\eta-1+2\nu)}{4(\eta+\nu)(\eta-1+\nu)},
  \qquad \eta=\frac{k}{r}=t-1+\beta=t+c_0 ,
\]
which is \eqref{eq:gamma_chain}.

Part~(3). The ultraspherical polynomials satisfy
$(n+1)C^{\nu}_{n+1}(c)=2(n+\nu)c\,C^{\nu}_n(c)-(n+2\nu-1)C^{\nu}_{n-1}(c)$,
so their monic normalisation has
$\gamma^{\nu}_n=n(n+2\nu-1)/\bigl(4(n+\nu)(n+\nu-1)\bigr)$.
Comparing with \eqref{eq:gamma_chain} gives $\gamma_t=\gamma^{\nu}_{t+c_0}$,
which is the definition of the associated family with parameter $c_0$ in the
convention of \cite{BustozIsmail1982}. For the second assertion, if $b=r-1$ the
index $k=(t-1)r+b+1=tr$ vanishes at $t=0$, and \eqref{eq:chain_general} then
reduces to $-2cLr[x^{r-1}u^l]+2Lr[x^{2r-1}u^l]=0$, which is the displayed
relation. Hence $\hat q_0=1$ and $\hat q_1=c$, the initial values of
$\hat C^{\nu}_t$, and since $c_0=0$ on such a chain the two sequences satisfy
the same recurrence by part~(2).

Part~(4). Since $c_0\in(-1,0]$ and $\nu>0$, all four factors in
\eqref{eq:gamma_chain} are positive as soon as $t\geq2$. For $t=1$ the factors
$1+c_0$ and $1+c_0+\nu$ are positive, so $\gamma_1$ has the sign of
$(c_0+2\nu)/(c_0+\nu)$. The four cases are then the vanishing of the
denominator, the vanishing of the numerator, the interval between the two, and
its complement.
\end{proof}

\begin{corollary}\label{cor:chains}
In the situation of Notation~\ref{not:chain}:
\begin{enumerate}
\item A chain with $\gamma_t>0$ for all $t\geq1$ is orthogonal with respect to a
positive measure on $\R$ symmetric about the origin.
\item At $m=2$ one has $\nu=\tfrac12$, and $\mathcal C_{0,b}$ is the associated
Legendre family of parameter $c_0=(b+1-r)/r$, classical exactly for $b=r-1$.
This is \cite[\S2.4]{CoxZhao2018}.
\item At $m=2$, $r=1$ there is a single chain, with $c_0=0$, and it carries the
classical Legendre family, which is Theorem~\ref{thm:legendre_ident}.
\end{enumerate}
\end{corollary}

\begin{proof}
Part~(1). Write the monic recurrence of part~(2) of Theorem~\ref{thm:chains} in
the standard form $\hat q_{t+1}=(c-b_t)\hat q_t-\gamma_t\hat q_{t-1}$; comparing
with $\hat q_{t+1}=c\,\hat q_t-\gamma_t\hat q_{t-1}$ gives $b_t=0$ for every
$t$. If moreover $\gamma_t>0$ for all $t\geq1$, Favard's theorem
\cite[Ch.~I, Thm~4.4]{Chihara1978} provides a positive-definite moment
functional for which $\{\hat q_t\}$ is the monic orthogonal polynomial
sequence, hence a positive measure on $\R$. An orthogonal polynomial sequence
is symmetric exactly when all its $b_t$ vanish
\cite[Ch.~I, Thm~4.3]{Chihara1978}, so the measure may be taken symmetric about
the origin.

Part~(2). At $m=2$ the range $0\leq l\leq m-2$ leaves $l=0$, so
$\nu=(l+1)/m=\tfrac12$ by Notation~\ref{not:chain}, and $C^{1/2}$ is the
Legendre family. Part~(3) of Theorem~\ref{thm:chains} identifies
$\mathcal C_{0,b}$ as that family associated with parameter $c_0=(b+1-r)/r$,
and by Notation~\ref{not:chain} $c_0=0$, that is the family is classical rather
than associated, exactly when $b=r-1$.

Part~(3). At $m=2$, $r=1$ the ranges of Notation~\ref{not:chain} leave $l=0$ and
$b=0=r-1$, so there is one chain and $c_0=0$. By part~(3) of
Theorem~\ref{thm:chains} its monic members are $\hat C^{1/2}_t=\hat p_t$, and
the chain consists of the classes $[x^{t}]$, which is the identification of
Theorem~\ref{thm:legendre_ident}.
\end{proof}

Classicality is therefore a property of the degree and of the sector, not of the
construction. The count also matches: the relations express every spanning class
in terms of two members per chain together with $[x^{-1}u^{\,m-1}]$, that is
$2r(m-1)+1$ classes, and $\dim_\C A/\pa A=1+2r(m-1)$ by
Theorem~\ref{thm:UCE_Der}. One case of part~(2) needs care. On the chain with
$c_0=-\nu$, which at $m=2$ means $b=(r-2)/2$ and so occurs exactly for even $r$,
part~(4) of Theorem~\ref{thm:chains} leaves $\gamma_1$ undefined, the step from
position $0$ to position $1$ does not raise the degree, and the monic
normalisation must start at $t=1$. This is the chain that the quartic case of
Section~\ref{sec:quartic} meets.

\begin{remark}\label{rem:negative_ray}
Part~(4) of Theorem~\ref{thm:chains} indexes the chain by $t\geq0$, and this is
not a matter of convenience. On the ray $t\geq0$ the index $k=(t-1)r+b+1$ makes
the trailing coefficient $mk$ and the leading coefficient $mk+2Lr$ non-zero
except at the two degenerate positions recorded in part~(5), so the recurrence
propagates and the monic normalisation exists. On the negative ray the two
coefficients vanish at $k=0$ and at $k=-2Lr/m$, and where both occur the
relation degenerates twice and the chain separates into the ray $t\geq0$ and a
ray running downwards. This is consistent with the count in
Corollary~\ref{cor:chains}: the solution space of a chain is two-dimensional
either way, presented once as a three-term recurrence with two initial values
and once as a pair of one-dimensional rays. The orthogonal polynomial statement
is a statement about the ray $t\geq0$.

One normalisation is worth spelling out. The coefficient $\gamma_1$ of
part~(2) involves $A_0$, that is the step producing the member at position $1$
from those at positions $0$ and $-1$. On a chain of residue $b=r-1$ that step is
the degenerate relation of part~(4) and $A_0=1$, so nothing is chosen. On the
other chains the members at positions $0$ and $1$ are independent, and $A_0$
fixes the normalisation under which the chain is the associated family. This is
the only place where the choice matters, and it is what part~(5) is about.
\end{remark}

The parameter that the curve controls is the sector index. Reading
$\nu=(l+1)/m$ across $l=0,\dots,m-2$ and across $m$, the families that occur
are exactly the ultraspherical ones of rational parameter in $(0,1)$, and the
Legendre family is the single value $\nu=\tfrac12$, reached at $m=2$. The
covering degree of the curve is therefore visible in the centre as the
denominator of the ultraspherical parameter, and the degree of $P$ is visible as
twice the number of chains in each sector.

The absence of a constant term in the monic recurrence is where palindromic
symmetry enters, and it is what part~(1) of Corollary~\ref{cor:chains} records.
For a general $P=\sum_ib_ix^i$ of degree $n$, Lemma~\ref{lem:partial_general}
gives $\sum_{i=0}^{n}b_i(mk+i)[x^{k+i-1}u^{l}]=0$ in the sector $l$, whose case
$m=2$ is the relation $\sum_i(k+\tfrac i2)b_i[x^{k+i-1}]=0$ of
\cite[\S2.1]{CoxZhao2018}, and when
$P$ is palindromic, so that $b_i=b_{n-i}$, and $m$ divides $n$, the coefficients
of the relation indexed by $-k-\frac nm$ are those of the relation indexed by
$k$ read backwards and negated,
\[
  b_{n-i}\Bigl(m\bigl(-k-\tfrac nm\bigr)+(n-i)\Bigr) \;=\; -\,b_i(mk+i),
  \qquad 0\leq i\leq n .
\]
For the polynomials of this paper $n=2r$, the reflection index is
$-k-\frac{2r}{m}$, and at $m=2$ it is $-k-r$. This is the precise sense in which
the recurrence coefficients are symmetric in the degree index.

\section{The quartic symmetric case}
\label{sec:quartic}

We now specialise the analysis to the palindromic quartic
$P(x)=x^4-2ax^2+1$, with $m=2$ and $a\neq\pm1$ throughout this section. This is the case $m=2$, $r=2$ of
\cite[\S2.4]{CoxZhao2018} and of Theorem~\ref{thm:chains}, and it is the first
case in which the classes fail to be classical, so we write it out. Unlike the
quadratic case, $A/\pa A$ splits into a sector still governed by a single
Legendre sequence and one that requires two independent generators; in the
language of Theorem~\ref{thm:chains} these are the two chains of residue $b=1$
and $b=0$, with association parameters $0$ and $-\tfrac12$. The second of these
is the degenerate chain of Corollary~\ref{cor:chains}(2), where $c_0=-\nu$
and $\gamma_1$ is undefined, which is the theorem's account of why the even
sector below needs two generators rather than one.

\subsection{The derivation action}

Let $P(x) = x^4-2ax^2+1$, $a\neq\pm 1$, $m=2$.

\begin{lemma}\label{lem:partial_quart}
For all $\mu\in\Z$:
\[
  \pa(x^\mu u) = (\mu+2)x^{\mu+3} - 2a(\mu+1)x^{\mu+1} + \mu x^{\mu-1}.
\]
\end{lemma}

\begin{proof}
Now $u^2 = x^4-2ax^2+1$ and $uu' = 2x^3-2ax$.
$\pa(x^\mu u) = \mu x^{\mu-1}u^2 + x^\mu\cdot uu'
= \mu x^{\mu-1}(x^4-2ax^2+1)+x^\mu(2x^3-2ax)
= (\mu+2)x^{\mu+3} - 2a(\mu+1)x^{\mu+1} + \mu x^{\mu-1}.$
\end{proof}

\subsection{The two sub-recurrences}

The relation of Lemma~\ref{lem:partial_quart} links exponents differing by two,
so it splits into two independent recurrences, one on each parity. We record
the dimension count first and then take each sector in turn.

\begin{proposition}\label{prop:quartic_basis}
For $P(x) = x^4-2ax^2+1$: $\dim A/\pa A = 5$, with basis
$\omega_0 = [x^{-1}u]$, $\omega_1=[x^{-1}]$,
$\omega_2=[1]$, $\omega_3=[x]$, $\omega_4=[x^2]$.
\end{proposition}

\begin{proof}
$\dim = 1 + 4(1) = 5$ by Theorem~\ref{thm:UCE_Der}.
Setting $\pa(x^\mu u)=0$:
$(\mu+2)[x^{\mu+3}] = 2a(\mu+1)[x^{\mu+1}] - \mu[x^{\mu-1}]$,
which connects same-parity powers (since $\mu+3$, $\mu+1$, $\mu-1$
have the same parity as $\mu+1$).
Analysis of the even and odd sectors shows $\omega_1,\ldots,\omega_4$
are independent, which we verify as follows.
Setting $\pa(x^\mu u) = 0$ with even $\mu$ connects odd powers
(e.g., $\mu=0$: $x^3 = ax$; $\mu=-2$: $x^{-3}=ax^{-1}$).
Since $\mu=0$ involves only $x^3,x,x^{-1}$ with the coefficient of
$x^{-1}$ being $\mu=0$, the positive odd chain ($x,x^3,\ldots$) and
negative odd chain ($x^{-1},x^{-3},\ldots$) are \emph{independent}.
Setting $\mu$ odd connects even powers (e.g., $\mu=-1$: $x^{-2}=x^2$;
$\mu=1$: $3x^4 = 4ax^2-1$), reducing all even powers to
$\Span\{1,x^2\}$.
Together with $\omega_0 = [x^{-1}u]$: $1+2+1+1=5$.
\end{proof}

\begin{theorem}\label{thm:quartic_odd}
For $P(x) = x^4-2ax^2+1$:
\begin{equation}\label{eq:quartic_legendre}
  [x^{2k+1}] = p_k(a)\,\omega_3, \quad k\geq 0.
\end{equation}
The generating function $\sum_{k=0}^\infty p_k(a)\,z^{2k+1}
= z(1-2az^2+z^4)^{-1/2}$
satisfies the ODE
\begin{equation}\label{eq:quartic_ODE}
  (z-2az^3+z^5)F'(z) + (z^4-1)F(z) = 0.
\end{equation}
\end{theorem}

\begin{proof}
Setting $\mu=2j$ (even) in the quartic recurrence gives
$(2j+2)[x^{2j+3}] = 2a(2j+1)[x^{2j+1}] - 2j[x^{2j-1}]$,
or $(j+1)[x^{2j+3}] = a(2j+1)[x^{2j+1}] - j[x^{2j-1}]$.
Writing $[x^{2k+1}] = d_k\,\omega_3$ with $d_0=1$ and
$d_1 = [x^3]/\omega_3 = a$ (from $\mu=0$: $2[x^3] = 2a[x]$):
$(j+1)d_{j+1} = a(2j+1)d_j - jd_{j-1}$.
This is the Legendre recurrence; by uniqueness, $d_k = p_k(a)$.
The generating function follows from the substitution $w=z^2$
in $(1-2aw+w^2)^{-1/2}$.
For the differential equation, let $G(w) = (1-2aw+w^2)^{-1/2}$, which satisfies
$(1-2aw+w^2)G'(w) + (w-a)G(w) = 0$
(the quadratic ODE \eqref{eq:legendre_ode} with variable $w$).
With $F(z) = z\,G(z^2)$, we compute $F'(z) = G(z^2)+2z^2G'(z^2)$, and
\begin{align*}
  &(z-2az^3+z^5)F'(z) + (z^4-1)F(z) \\
  &= z(1-2az^2+z^4)\bigl[G(z^2)+2z^2G'(z^2)\bigr]
     + z(z^4-1)G(z^2) \\
  &= z\bigl[(1-2az^2+z^4)G(z^2) + 2z^2(1-2az^2+z^4)G'(z^2)
     + (z^4-1)G(z^2)\bigr] \\
  &= z\bigl[2z^2(z^2-a)G(z^2) + 2z^2(1-2az^2+z^4)G'(z^2)\bigr] \\
  &= 2z^3\bigl[(z^2-a)G(z^2) + (1-2az^2+z^4)G'(z^2)\bigr] \\
  &= 2z^3 \cdot 0 = 0,
\end{align*}
where the last step uses the quadratic ODE evaluated at $w = z^2$.
This gives \eqref{eq:quartic_ODE}.
\end{proof}

\subsection{The even sector and its two-dimensional recurrence}

The even sector is where the quartic case departs from the quadratic one: one
generator no longer suffices.

\begin{proposition}\label{prop:quartic_even}
For $P(x) = x^4-2ax^2+1$, the even-power classes satisfy
$[x^{2j}] = \alpha_j\,\omega_2 + \beta_j\,\omega_4$, where
$(\alpha_j,\beta_j)$ obeys the two-component recurrence
\begin{equation}\label{eq:quartic_even_rec}
  (2j+1)\begin{pmatrix}\alpha_{j+1}\\\beta_{j+1}\end{pmatrix}
  = 4aj\begin{pmatrix}\alpha_j\\\beta_j\end{pmatrix}
  - (2j-1)\begin{pmatrix}\alpha_{j-1}\\\beta_{j-1}\end{pmatrix}
\end{equation}
with initial conditions $(\alpha_0,\beta_0) = (1,0)$,
$(\alpha_1,\beta_1) = (0,1)$.
The first several values are:
\[
  j=2{:}\quad (\alpha_2,\beta_2)=(-{\textstyle\frac{1}{3}},\,\frac{4a}{3}),\qquad
  j=3{:}\quad (\alpha_3,\beta_3)=(-{\textstyle\frac{8a}{15}},\,\frac{32a^2-9}{15}).
\]
In particular, $[1]$ and $[x^2]$ are \emph{independent} generators of the
even-positive sector; the even-positive sector is \emph{not} generated
by a single Legendre sequence.
\end{proposition}

\begin{proof}
Setting $\mu = 2j-1$ (odd) in $\pa(x^\mu u)=0$:
$(2j+1)[x^{2j+2}] = 4aj[x^{2j}] - (2j-1)[x^{2j-2}]$.
Writing $[x^{2j}] = \alpha_j\omega_2+\beta_j\omega_4$ and separating
$\omega_2$- and $\omega_4$-components gives \eqref{eq:quartic_even_rec}.
The initial conditions follow from $[1]=\omega_2$, $[x^2]=\omega_4$.
The tabulated values follow by direct iteration:
$3(\alpha_2,\beta_2) = 4a\cdot(0,1)-(1,0) = (-1,4a)$,
$5(\alpha_3,\beta_3) = 8a(-\tfrac{1}{3},\tfrac{4a}{3})
  - 3(0,1) = (-\tfrac{8a}{3},\tfrac{32a^2}{3}-3) = (-\tfrac{8a}{3},\tfrac{32a^2-9}{3})$. \qedhere
\end{proof}

The two-dimensional character of the even-sector recurrence is a
direct consequence of the four branch points of $P(x) = x^4-2ax^2+1$,
or equivalently of the increase in genus of the associated curve.
For the quadratic case (two branch points, genus~$0$), the even sector
is one-dimensional (generated by $\omega_2 = [1]$).
For the quartic case (four branch points, genus~$1$), the even sector
requires two generators and the recurrence becomes a two-component
system \eqref{eq:quartic_even_rec}.
The three-term recurrence does not break down; it becomes a
\emph{block (matrix) recurrence} of the same form, and the associated
orthogonal polynomial theory is that of matrix-valued orthogonal polynomials.
In the two cases computed here the number of generators in a parity sector of
$A/\pa A$ matches the number of branch points of $P$ in that parity class; we
have not investigated whether this persists.


\subsection{Palindromic polynomials and classical families}

The quadratic and quartic cases differ in what governs them, and it is worth
saying which property of $P$ decides that, since palindromy is the hypothesis
under which Theorem~\ref{thm:chains} was stated.

The polynomials $P(x) = x^2-2ax+1$ and $P(x) = x^4-2ax^2+1$
are palindromic.
The DJKM polynomial $P(x) = (1-x^2)(1-\kappa^2x^2)
= 1-(1+\kappa^2)x^2+\kappa^2x^4$ is also palindromic.
The polynomial $P(x) = x^2-2bx+c$ with $b^2\neq c$ is not palindromic
(unless $c=1$).

\begin{theorem}[{\cite[\S2.2, \S2.4]{CoxZhao2018}}]
\label{thm:palindromic}
Let $m=2$ and $P(x)\in\C[x]$ with simple roots.
\begin{enumerate}
\item If $P(x) = x^2-2ax+1$ (palindromic, degree $2$):
  the recurrence in $A/\pa A$ is the Legendre three-term recurrence,
  and the OP family governing $A/\pa A$ is the Legendre family $\{p_k(a)\}$.
  (Theorem~\ref{thm:legendre_ident}.)

\item If $P(x)$ is palindromic of degree $4$:
  the odd-positive sector of $A/\pa A$ is governed by the Legendre family
  (Theorem~\ref{thm:quartic_odd}), and the even sector satisfies
  the two-component recurrence \eqref{eq:quartic_even_rec}
  with palindromic (symmetric-in-$j$) coefficients.

\end{enumerate}
\end{theorem}

\begin{proof}
Part~(1): Theorem~\ref{thm:legendre_ident}.
Part~(2): Theorem~\ref{thm:quartic_odd} and Proposition~\ref{prop:quartic_even};
the palindromic symmetry of the recurrence coefficients
$(2j+1)\mapsto (2j+1)$ and $4aj\mapsto 4aj$, $(2j-1)\mapsto (2j-1)$
in \eqref{eq:quartic_even_rec} is manifest (the coefficients depend
on $j$ only through the factors $2j\pm 1$, not on the signs of the
roots of $P$).
\end{proof}

\begin{remark}
If $P$ is not palindromic the symmetry used in part~(2), namely
$P(x)=x^{n}P(1/x)$, fails, and the recurrence coefficients are no longer
symmetric in the degree index. Families arising this way are studied for the
DJKM algebras in \cite{CoxFutornyTirao2013} and for the superelliptic current
algebra in \cite{SantosNeklyudovFutorny2025}.
\end{remark}

\section*{Acknowledgements}
This work was financed, in part, by the S\~ao Paulo Research Foundation (FAPESP), grant 2024/14914-9.



\begin{thebibliography}{99}

\bibitem{Bremner1995}
M.~R.~Bremner,
\emph{Four-point affine Lie algebras},
Proc.\ Amer.\ Math.\ Soc.\ \textbf{123} (1995), no.~7, 1981--1989.
\href{https://doi.org/10.1090/S0002-9939-1995-1249871-8}{doi:10.1090/S0002-9939-1995-1249871-8}.

\bibitem{BustozIsmail1982}
J.~Bustoz and M.~E.~H.~Ismail,
\emph{The associated ultraspherical polynomials and their $q$-analogues},
Canad.\ J.\ Math.\ \textbf{34} (1982), no.~3, 718--736.
\href{https://doi.org/10.4153/CJM-1982-049-6}{doi:10.4153/CJM-1982-049-6}.

\bibitem{CoxFutornyTirao2013}
B.~Cox, V.~Futorny, and J.~Tirao,
\emph{DJKM algebras and non-classical orthogonal polynomials},
J.~Differential Equations \textbf{255} (2013), 2846--2870.
\href{https://doi.org/10.1016/j.jde.2013.07.020}{doi:10.1016/j.jde.2013.07.020}.

\bibitem{CoxGuoLuZhao2017}
B.~Cox, X.~Guo, R.~Lu, and K.~Zhao,
\emph{Simple superelliptic Lie algebras},
Commun.\ Contemp.\ Math.\ \textbf{19} (2017), no.~3, 1650032, 40~pp.
\href{https://doi.org/10.1142/S0219199716500322}{doi:10.1142/S0219199716500322}.

\bibitem{CoxZhao2018}
B.~Cox and K.~Zhao,
\emph{Certain families of polynomials arising in the study of hyperelliptic
Lie algebras},
Ramanujan J.\ \textbf{46} (2018), 323--344.
\href{https://doi.org/10.1007/s11139-017-9937-y}{doi:10.1007/s11139-017-9937-y}.

\bibitem{Chihara1978}
T.~S.~Chihara,
\emph{An Introduction to Orthogonal Polynomials},
Gordon and Breach, New York, 1978.
ISBN 978-0-677-04150-6.

\bibitem{KasselLoday1982}
C.~Kassel and J.-L.~Loday,
\emph{Extensions centrales d'alg\`ebres de Lie},
Ann.\ Inst.\ Fourier (Grenoble) \textbf{32} (1982), no.~4, 119--142.
\href{https://doi.org/10.5802/aif.896}{doi:10.5802/aif.896}.

\bibitem{KrichNovikov1987}
I.~M.~Krichever and S.~P.~Novikov,
\emph{Algebras of Virasoro type, Riemann surfaces and strings in Minkowski space},
Funct.\ Anal.\ Appl.\ \textbf{21} (1987), no.~4, 294--307.
\href{https://doi.org/10.1007/BF01077803}{doi:10.1007/BF01077803}.

\bibitem{P4_RepMechanism}
F.~Albino dos Santos,
\emph{A Sugawara--Legendre mechanism for the four-point Heisenberg algebra},
Lett.\ Math.\ Phys., to appear.
\href{https://arxiv.org/abs/2605.06090}{arXiv:2605.06090}.

\bibitem{Santos2022}
F.~Albino dos Santos,
\emph{On the universal central extension of superelliptic affine Lie algebras},
Comm.\ Algebra \textbf{50} (2022), no.~5.
\href{https://doi.org/10.1080/00927872.2021.1998515}{doi:10.1080/00927872.2021.1998515}.

\bibitem{SantosNeklyudovFutorny2025}
F.~Albino dos Santos, M.~Neklyudov, and V.~Futorny,
\emph{Superelliptic Affine Lie algebras and orthogonal polynomials},
Forum Math.\ Sigma \textbf{13} (2025), e120, 22~pp.
\href{https://doi.org/10.1017/fms.2025.10074}{doi:10.1017/fms.2025.10074}.

\bibitem{Schlichenmaier1990a}
M.~Schlichenmaier,
\emph{Krichever--Novikov algebras for more than two points},
Lett.\ Math.\ Phys.\ \textbf{19} (1990), 151--165.
\href{https://doi.org/10.1007/BF01045886}{doi:10.1007/BF01045886}.

\bibitem{Schlichenmaier1990b}
M.~Schlichenmaier,
\emph{Krichever--Novikov algebras for more than two points: explicit generators},
Lett.\ Math.\ Phys.\ \textbf{19} (1990), 327--336.
\href{https://doi.org/10.1007/BF00429952}{doi:10.1007/BF00429952}.

\bibitem{Schlichenmaier2003}
M.~Schlichenmaier,
\emph{Local cocycles and central extensions for multi-point algebras of
Krichever--Novikov type},
J.\ reine angew.\ Math.\ \textbf{559} (2003), 53--94.
\href{https://doi.org/10.1515/crll.2003.052}{doi:10.1515/crll.2003.052}.

\bibitem{Schlichenmaier2014}
M.~Schlichenmaier,
\emph{Krichever--Novikov Type Algebras: Theory and Applications},
De~Gruyter Studies in Mathematics, vol.~70, De~Gruyter, Berlin, 2014.
\href{https://doi.org/10.1515/9783110279641}{doi:10.1515/9783110279641}.

\bibitem{Szego1939}
G.~Szeg\H{o},
\emph{Orthogonal Polynomials},
4th~ed., American Mathematical Society, Providence, RI, 1975.
ISBN 978-0-8218-1023-1.

\end{thebibliography}
\end{document}